\documentclass[dvips]{article}
\usepackage{amsmath,amssymb,amsfonts,amsthm,epsfig,latexsym,graphics}%,makeidx}
\usepackage[usenames,dvipsnames]{color}
 \def\FF{{\mathbb F}}  
   
 \def\NN{{\mathbb N}}  
 \def\RR{{\mathbb R}}  
   
 \def\ZZ{{\mathbb Z}}

\def\Si{\Sigma}
\def\La{\Lambda}

\def\Om{\Omega}
\def\Ga{\Gamma}

\def\cA{{\cal A}}  \def\cG{{\cal G}} \def\cM{{\cal M}} \def\cS{{\cal S}}
\def\cB{{\cal B}}  \def\cH{{\cal H}}  
\def\cC{{\cal C}}  \def\cI{{\cal I}}  
   \def\cP{{\cal P}} 
\def\cE{{\cal E}}

\newtheorem{theo}{Theorem}[section]

\newtheorem{clai}{Claim}

\newtheorem{lemm}{Lemma}[section]
\newtheorem{coro}[lemm]{Corollary}

\newtheorem{defi}[lemm]{Definition}
\newtheorem{prop}[lemm]{Proposition}

\newtheorem{ques}{Question}
\newtheorem{nota}{Notation}

\newtheorem{rema}[lemm]{Remark}

\newenvironment{demo}[1][Proof]{\noindent {\bf #1~: }}{\hfill$\Box$\medskip}

\theoremstyle{definition}

\theoremstyle{remark}

\usepackage[french]{babel}

\usepackage{amsmath}
\usepackage{amsfonts}
\usepackage{amsthm}
\usepackage{amssymb}
\usepackage{mathrsfs}
\usepackage{graphicx}

\newcommand{\limite}[2]{\mathop{\longrightarrow}
\limits_{\mathrm{#1}}^{\mathrm{#2}}}

\newcommand{\N}{\mathbb{N}}
\newcommand{\Z}{\mathbb{Z}}

\newcommand{\R}{\mathbb{R}}

\newcommand{\mc}[1]{\mathcal{#1}}

\newcommand{\crochetl}{\left[\kern-0.15em\left[}
\newcommand{\crochetr}{\right]\kern-0.15em\right]}

\def\restriction#1#2{\mathchoice
              {\setbox1\hbox{${\displaystyle #1}_{\scriptstyle #2}$}
              \restrictionaux{#1}{#2}}
              {\setbox1\hbox{${\textstyle #1}_{\scriptstyle #2}$}
              \restrictionaux{#1}{#2}}
              {\setbox1\hbox{${\scriptstyle #1}_{\scriptscriptstyle #2}$}
              \restrictionaux{#1}{#2}}
              {\setbox1\hbox{${\scriptscriptstyle #1}_{\scriptscriptstyle #2}$}
              \restrictionaux{#1}{#2}}}
\def\restrictionaux#1#2{{#1\,\smash{\vrule height .8\ht1 depth .85\dp1}}_{\,#2}}

\title{Centralizers of $C^1$-contractions of the half line}
\author{C. Bonatti and \'E. Farinelli}
\date{\today}

\begin{document}
\maketitle

{\footnotesize
\noindent 
{\bf Abstract:} 
{\it A subgroup  $G\subset Diff^1_+([0,1])$ is  
\emph{\bf{$C^1$-close to the identity}} if  there is a sequence $h_n\in Diff^1_+([0,1])$ such that the conjugates
$h_n g h_n^{-1}$ tend to the identity for the $C^1$-topology, for every $g\in G$. 
This is equivalent to the fact that 
 $G$  can be \emph{\bf{embedded in the $C^1$-centralizer}} of a $C^1$-contraction of 
 $[0,+\infty)$(see \cite{Fa} and 
 Theorem~\ref{t.cohomologique}). 

We first describe the topological dynamics of groups $C^1$-close to the identity. Then, we show that 
the class of groups $C^1$-close to the 
identity is invariant under some natural dynamical and algebraic extensions.
As a consequence,  we can describe a large class of groups $G\subset Diff^1_+([0,1])$
whose topological dynamics implies that 
they are $C^1$-close to the identity.

This allows us to show that the free group $\FF_2$ admits faithfull 
actions which are $C^1$-close to the identity. In particular, the $C^1$-centralizer of  a $C^1$-contraction
may contain free groups.

}
\vskip 5mm


}

\section{Introduction}
\subsection{Groups $C^1$-close to the identity and centralizers of contractions}
The main  motivation of this paper is the study of centralizers of the \textbf{\emph{$C^1$-contractions 
of the half line $[0,+\infty)$}}. A diffeomorphism $f$ of 
$[0,+\infty)$ is called a contraction if $f(x)<x$ for every $x\neq 0$. 
Unless it is explicitely indicated, a contraction will now refer to a $C^1$-diffeomorphism.
When $f$ is at least $C^2$, 
 Szekeres, in \cite{Sz} (see also \cite{Se}),
proved that $f$ is the time-one map of the flow of a $C^1$-vector field $X$, and Kopell's Lemma (see \cite{Ko}) implies that the $C^1$-centralizer of $f$ is 
precisely the flow $\{X_t, t\in\RR \}$ of the Szekeres vector field. When $f$ is only required to be $C^1$, Szekeres result does not hold anymore
and neither does Kopell's Lemma. Actually, the $C^1$-centralizer of a $C^1$-contraction $f$ may be very different according to $f$.  Generically it
is trivial (i.e. equal to $\{f^n,n\in\ZZ\}$, see \cite{To}) but it can also be very large (non abelian and non countable). 
We will see that there are at the same time many algebraic and dynamical restrictions on the possible groups, but also a 
large variety of dynamical properties which allows a group to be included in a centralizer of a contraction.

In \cite{Fa}, we consider groups $G$ of diffeomorphisms of a segment $I\subset (0,1)$. We say  that 
$G$ is \textbf{\emph{embeddable in the centralizer of a contraction}} if  there exists a contraction $f$ of $[0,+\infty)$ 
and a subgroup  $\tilde G$ of the $C^1$-centralizer of $f$ which induces $G$ by restriction to $I$.  
\cite[Theorem 3]{Fa} shows that $G$ is embeddable in the centralizer of a contraction if and only if there is a $C^1$-continuous path of 
diffeomorphisms $h_t\in Diff^1_+(I)$ such that $h_t g h_t^{-1}$ tends to the identity for every $g\in G$.

As a direct consequence one deduces that, if $G$ is embeddable 
in the centralizer of a contraction, then
$G$ is also \textbf{\emph{embeddable in the centralizer of a diffeomorphism 
$f$ of $[0,1]$ without fixed point in $(0,1)$}}.  

 Finally,  an argument by A. Navas proves that:

\begin{theo}\label{t.cohomologique} Let $I$ be a compact segment. 
Given a group $G\subset Diff^1(I)$, the two following properties are equivalent:
\begin{itemize}
\item  there is a $C^1$-continuous path of 
diffeomorphisms $h_t\in Diff^1_+(I)$, $t\in[0,1)$, such that $h_t g h_t^{-1}$ tends to the identity in the $C^1$-topology when $t$ tends to $1$, for every $g\in G$;
\item there is a sequence of 
diffeomorphisms $h_n\in Diff^1_+(I)$, $n\in\NN$,  such that $h_n g h_n^{-1}$ tends to the identity  in the $C^1$-topology, for every $g\in G$ as $n\to +\infty$.
\end{itemize}
\end{theo}
This fact is not trivial at all and is specific to the identity map:  \cite{Fa} provides 
examples of pairs of diffeomorphisms 
$f,g\in Diff^1([0,1])$ such that there are sequences $h_n\in Diff^1([0,1])$ leading $g$  to $f$ by conjugacy
(that is $h_n g h_n^{-1}\underset{n\to\infty}{\overset{C^1}{\longrightarrow}} f$), but such that there is no continous path $h_t$ leading $g$ to $f$ by conjugacy.
The proof of Theorem \ref{t.cohomologique} is presented in Section~\ref{s.cohomologique}.

Therefore, we have four equivalent notions which induce a well defined class of subgroups $G$ of 
$Diff^1_+([0,1])$:

\begin{itemize}
\item $G$ is embeddable in the centralizer of a contraction;
\item $G$ is embeddable in the centralizer of a diffeomorphism $f\in\mc{D}iff^1_+(I)$ without fixed point in the interior of $I$;
\item there exists a path of diffeomorphisms $h_t\in\mc{D}iff^1_+([0;1])$ leading by conjugacy 
every  $g\in G$ to the identity;
\item  there exists a sequence of diffeomorphisms $h_n\in\mc{D}iff^1_+([0;1])$ leading by conjugacy every  $g\in G$ to the identity.
\end{itemize}

\begin{defi} Let $I$ be a segment. 
We say that a group $G\subset Diff^1_+([0,1])$ is \emph{\textbf{$C^1$-close to the identity}} if it 
satisfies one of the four equivalent properties above, that is, for instance, if there is a sequence of 
diffeomorphisms $h_n\in Diff^1_+([0;1])$ such that for every $g\in G$ $$h_n g h_n^{-1}\underset{n\to\infty}{\overset{C^1}{\longrightarrow}}id.$$
\\
We denote by \textbf{$\mc{C}^1_{id}(I)$} the class consisting of these groups; when $I=[0,1]$ we simply denote it by
$\mc{C}^1_{id}$. 
\end{defi}

The aim of this paper is to describe this class of groups $C^1$-close to identity, 
up to group isomorphisms, up to topological conjugacy, and/or  up to smooth conjugacy. 
\\
In other words, we try to answer to the following questions:

\begin{ques} What group $G$ admits a faithful representation $\varphi\colon  G\to Diff^1([0,1)$ so that 
$\varphi(G)$ is $C^1$-close to the identity? 
\end{ques}

As a partial answer, \cite{Mc} implies that Baumslag-Solitar groups cannot be $C^1$-close to the identity. 
In contrast,Theorem~\ref{t.free} shows  that the free group $\FF_2$ admits actions $C^1$-close 
to the identity. 

\begin{ques} What is the topological dynamics of a group $\mc{C}^1$-close to the identity? 
\\
In other words, given  a group $G\subset Homeo_+([0,1))$, under what hypotheses does there exist $h\in Homeo_+([0,1])$ 
such that  $hGh^{-1}$ is contained in $\cC^1_{id}\subset Diff^1_+([0,1])$?
\end{ques}
Theorem~\ref{t.realisation} presents a large class of groups $G\subset Diff^1_+([0,1])$, called 
\textbf{\emph{elementary groups}}, whose 
topological dynamics implies that they are $C^1$-close to the identity. Thus, every group $G'\subset Diff^1_+([0,1])$ 
topologically conjugated to $G$ is $C^1$-close to the identity. 

\begin{ques} Given a group $G\subset Diff^1_+([0,1])$, under what conditions does it belong to  $\mc{C}^1_{id}$ ?
 
\end{ques}

Let us illustrate this question by another which is more precise.
An immediate obstruction
for a group $G\subset Diff^1_+([0,1])$ to be $C^1$-close to the identity is the existence of a hyperbolic fixed point for an element $g\in G$.  
We will say that $G$ is \textbf{\emph{without hyperbolic fixed points}} if, for every $g\in G$ and every $x\in Fix(g)$, $Dg(x)=1$. 
A. Navas suggested that it could be a necessary and sufficient condition, 
maybe for finitely presented groups. \cite{Fa}  proved that this is true for cyclic groups. 

\begin{ques} Let $G$ be a subgroup of  $Diff^1_+([0,1])$  (maybe assuming countable, or finitely presented,
or any other natural hypothesis). We wonder if
\[  G \quad \mbox{without hyperbolic fixed point} \overset{?}{\Longleftrightarrow} G \quad C^1\mbox{-close to identity} \]
\end{ques}
We denote  by $\mc{C}^1_{nonhyp}$ the class of groups $G\subset Diff^1_+([0,1])$ without hyperbolic fixed points. For now we know:

\[ \mc{C}^1_{id}\subset\mc{C}^1_{nonhyp} \]

The rest of the introduction expounds the statements of the results above, and 
proposes some directions for answering these questions.

\subsection{Structure results: description of the topological dynamics}\label{s.topologicaldynamics}

This section gives necessary conditions on  the topological dynamics of a group $G\subset Diff^1_+([0,1])$, 
 for $G$ being  $C^1$-close to the identity.  
 
Given a group $G\subset Homeo^+([0,1])$, a \textbf{\emph{pair of successive fixed points 
of $G$}} is a pair $\{a,b\}$ such that $(a,b)$ is a 
connected component of $[0,1]\setminus Fix (g)$ for some $g\in G$. 
One says that two pairs of successive fixed points $\{a,b\}$ and $\{c,d\}$ are \textbf{\emph{linked}}
 if $(a,b)\cap\{c,d\}$ or $(c,d)\cap\{a,b\}$ consists in exactly one point. 

\begin{defi}\label{d.linked}
A group  $G\subset Homeo^+([0,1])$ is \textbf{\emph{without linked fixed points}} if 
there are no linked pairs 
of successive fixed points. 
\end{defi}
The main topological restriction  for a group $G\subset Diff^1_+([0,1])$ to be  $C^1$-close to identity is: 
 \begin{theo}\label{t.linked}
  If $G\subset Diff^1_+([0,1])$ is $C^1$-close to the identity, then $G$ is without linked  fixed points.
 \end{theo}
Using completely different methods due to A. Navas, one proves a slightly stronger result: 
 \begin{theo}\label{t.sanshyperbolique} (\cite{Na})  Any group $G\subset Diff_+^1([0,1])$  
 without hyperbolic fixed point 
 is without linked fixed points. 
 \end{theo}

Thus, the intervals of successive fixed points form a nested family. As a consequence, we will see that 
the family of pairs of successive fixed points is at most countable (Proposition~\ref{p.countable}).
 
Another consequence of beeing without linked fixed points is that, for every interval $(a,b)$ of 
successive fixed points and any $g\in G$, 
either $g([a,b])=[a,b]$ or $g([a,b])\cap (a,b)=\emptyset$. Considering  the stabilizers $G_{[a,b]}$
of the segments $[a,b]$, this provides   a stratified description of the dynamics of $G$, 
as stated in Theorem~\ref{t.structure} below.

 \begin{theo}\label{t.structure}\label{t.sanscroisement}Let $G\subset Homeo_+([0,1])$ be a group without linked fixed point, 
 and $\{a,b\}$  be a pair of successive fixed points. 
 
 Then:
 \begin{itemize} 
 \item for every $g\in G$, either $g([a,b])=[a,b]$ or $g((a,b))\cap (a,b)=\emptyset$. 
 We denote by $G_{[a,b]}$ \textbf{\emph{the stabilizer of $[a,b]$}}. 
  \item there is a morphisms $\tau_{a,b} : G_{[a,b]}\to \RR$ whose kernel is precisely the  set of $g\in G_{[a,b]}$ 
  having fixed points in $(a,b)$, and which is 
  positive at $g\in G_{[a,b]}$ if and only if $g(x)-x>0$ on $(a,b)$.
  \item The union of the minimal sets of the action of $G_{[a,b]}$ on $(a,b)$ is a non-empty closed subset $\La_{[a,b]}$ on 
  which the action of $G_{[a,b]}$ is 
 semi-conjugated to the group of translations $\tau_{a,b}(G_{[a,b]})$. 
 \item The elements of the kernel of $\tau_{a,b}$ induce the identity map on $\La_{a,b}$. 
 \end{itemize}

 The morphism $\tau_{a,b}$ is unique up to multiplication by a positive number. 
 \end{theo}  
The morphism $\tau_{a,b}$ is called a \textbf{\emph{relative translation number}}.
 
For describing the dynamics of $G$ we are lead to consider: 

 \begin{itemize}
  \item the nested configuration of the intervals of successive fixed points;
  \item for each interval $(a,b)$ of successive fixed points, the relative translation number $\tau_{a,b}$
  and the set $\La_{a,b}$,  union of the minimal sets in $(a,b)$.   Each connected component $I$ of 
  $(a,b)\setminus \La_{a,b}$ is 
  \begin{itemize}
  \item either a wandering interval
  \item or an interval of successive fixed points,
  \item  or else it may be the union  of 
  an increasing sequence of intervals  of successive fixed points. 
  \end{itemize}
  In the first case, one can stop the study.  In both last cases, one considers the restriction to $I$ of the stabilizer $G_{I}$: it is a group
  without linked fixed points, so one may proceed the study.  
 \end{itemize}

 %\vskip 2mm
%\subsection{Towards a topological presentation of groups $C^1$-close to the identity}

%As we said above, there are three main difficulties for classifiying the groups 
%$C^1$-close to the identity: 
% \begin{itemize}
%  \item the nested configuration of intervals of successive fixed points,
%  \item the translation number associated to each of these intervals,
%  \item each element $g\in G$ may have several or infinitely many pairs of successive fixed points.  
% \end{itemize}

 \subsection{Completion of a group without linked fixed points}
 One difficulty for classifying the groups $C^1$-close to the identity it that each element $g\in G$ may have infinitely many pairs 
 of successive fixed points.
 For bypassing this difficulty, we  enrich the group $G$ so that,
 for every $g\in G$ and any pair $\{a,b\}$ of successive fixed points of $g$, the group $G$ contains  the diffeomorphism $g_{a,b}$
  which coincides with $g$ on $[a,b]$ and with the identity out of $[a,b]$.  The diffeomorphism $g_{a,b}$ has a unique pair of successive fixed point. 
  Such a diffeomorphism will be called  \textbf{\emph{simple}}.
 Every element $g\in G$ can be seen as an infinite product of the  simple elements $g_{a,b}$, for all the pairs $\{a,b\}$   of 
 successive fixed points of $g$. 
 
 More precisely, given $g,h\in Homeo_+([0,1])$, we say that \textbf{\emph{$h$ is induced by $g$}} 
 if  $g$ and $h$  coincide on the support of $h$.  
 We say that a group $G\subset Homeo_+([0,1])$ without linked fixed points is \textbf{\emph{complete}} if it contains any homeomorphism
 $h$ induced by an element $g\in G$. 
 
Corollary~\ref{c.c1completion} shows that every group $C^1$-close to the identity is a subgroup of a complete group $C^1$-close to the identity.  
Analoguous results hold for groups of homeomorphisms without linked fixed points, or for groups of diffeomorphisms without hyperbolic fixed points: 
Proposition~\ref{p.completion} associates to each group $G$ without linked fixed points its \emph{\textbf{completion} $\tilde G$} which is the 
 smallest complete group without linked fixed points  
 containing $G$. The families of intervals of successive fixed points of $\tilde G$ and $G$ are the same, and 
 Corollary~\ref{c.translationcompletion}
states that the morphisms of translation numbers 
associated to each interval of successive fixed points also coincide for $G$ and $\tilde G$.

\subsubsection{Totally rational groups and topological basis}

Let us present here a problem, coming from an unsuccessful attempt of us for classifying groups $C^1$-close to the identity.

 We say that a group $G$ without linked fixed points is \textbf{\emph{totaly rational}} 
 if, for any successive fixed points $\{a,b\}$, the image of the translation number 
 $\tau_{a,b}(G_{[a,b]})$ is a cyclic (monogene) group (i.e. a group of the form $\alpha \ZZ$). 
 Proposition~\ref{p.completion} and Corollary~\ref{c.translationcompletion} 
 imply that the completion of a totally rational group is totally rational, 
 allowing us to consider complete totally rational groups.

 Given a complete totally rational group $G$, 
 let us call a \textbf{\emph{topological basis}} any family $\{f_i\}_{i\in\cI}$ of elements 
 $f_i\in G$ such that:
 \begin{itemize}
  \item for every $i\in\cI$, $f_i$ has a unique pair $\{a_i,b_i\}$ of successive fixed points; 
  thus $[a_i,b_i]$ is the support of $f_i$.
  \item for every $i\in\cI$, let $\tau_i$ be the relative translation number associated to $(a_i,b_i)$. 
  As $G$ is totally rational, its image is a cyclic group.
  We require that  $\tau_i(f_i)$  is a generator of the image of $\tau_i$. 
  \item for every pair $\{a,b\}$ of successive fixed points of $G$, there is a unique $i\in \cI$ such that  
  $\{a_i,b_i\}$ and $\{a,b\}$ are in the same $G$-orbit.

  \item for every pair $\{a,b\}$ of successive fixed points of $G$, there is an element $g$ of the subgroup 
  $<f_i,i\in\cI>$ generated by the $f_i$ such that $\{a,b\}$ is a pair of successive fixed points of $g$.
 \end{itemize}

 The family of intervals of successive fixed points is countable (Proposition~\ref{p.countable}), 
 so that any countable basis is countable. We can get rather easily the three first items of the definition: 
 just choose one interval in each $G$-orbit of intervals of successive fixed points and for this interval choose a
 generator of the corresponding translation number.  Thus, the difficulty is 
 the last item. We have not been able to solve it\footnote{Consider the group generated by a family $g_i\in Diff^1_+([0,1])$, so that 
 the support of $g_i$ consists in one interval contained in a fundamental domain of $g_{i+1}$. Now consider 
 the family $f_i=g_{i+1}g_ig_{i+1}^{-1}$; 
 this family satisfies the three first hypotheses but not the last one.}, and
 the following natural question  remains open: 
 
 \begin{ques} Given a complete totally rational group $G\subset Homeo_+([0,1])$ without linked fixed points, does $G$ admit a topological basis ? 
\\
If the answer is negative, same question with the more restrictive assumption that $G\subset Diff^1_+([0,1])$ is
 $C^1$-close to the identity.
 \end{ques}

Our intuition is that every element $g\in G$ is be determined by its intervals of 
 successive fixed points and on each of them, by the value of the corresponding translation number, which provides a coordinates in
 the topological basis. However, even assuming the existence of a topological basis, there remain many issues 
 before making rigorous 
 our intuition.  In particular:
 
 \begin{ques} Let $G\subset Homeo_+([0,1])$ without linked fixed points, and assume $G$ admits a topological basis 
 $\{f_i\}_{i\in \cI}$.  Is $G$ contained in the $C^0$-closure of the group generated by the $f_i$? 
 
 Same question, in the $C^1$-topology, with the more restrictive assumption that $G\subset Diff^1_+([0,1])$ is
 $C^1$-close to the identity.  
 \end{ques}

 %\subsubsection{The configuration of the intervals 
 %of successive fixed points}
 
 %In the next section, we  will focus on the first difficulty: 
 
 %\centerline{\emph{What are the possible configurations for the intervals 
 %of successive fixed points?}} 
%We will attack this problem from the other end: we will put combinatorial rules on a family 
% $\{f_i\}_{i\in \NN}$ of  diffeomorphisms so that  the group generated will be obviously without linked fixed points 
% (and totally rational and complete). Then the family  will be a topological basis of the generated group. Our aim is 
% to show that the generated group is $C^1$-close to the identity. Theorem~\ref{t.realisation} will show that all the configurations
% of pairs of successive fixed points are allowed for groups $C^1$-close to the identity.  

\subsection{Realisation results}

\subsubsection{Invariance of $\cC^1_{id}$ by some extensions.}
Let us first present two  results, enlighting some invariance of the class $\cC^1_{id}$ by some natural 
extensions.

\begin{theo}\label{t.union}
Let $G_n\subset Diff^1_+([0,1])$, $n\in\NN$, be an increasing sequence of  subgroups:  
$$\forall n\in\NN, G_n\subset G_{n+1}.$$
\\
Assume that every $G_n$ is finitely generated and is $C^1$-close to the identity. 
\\
Then  $G= \bigcup_{n\in\NN} G_n \subset Diff^1_+([0,1])$ is $C^1$-close to the identity. 
\end{theo}
We don't know if Theorem~\ref{t.union} holds for uncountable groups $G_n$.

The technical heart of this paper consists in proving: 
\begin{theo}\label{t.jump}  Consider a segment $I\subset (0,1)$.  Let $G\subset Diff^1_+(I)$ be a group
$C^1$-close to the identity  and $f\in Diff^1_+([0,1])$ be a diffeomorphism such that 
$$f(I)\cap I=\emptyset$$
(in other words, $I$ is contained in the interior of a fundamental domain of $f$).
\\
Then the group $<f,G>$ generated by $f$ and $G$ is $C^1$-close to the identity.
\end{theo}

Actually, we will prove in Theorem~\ref{t.jump2}  a slightly stronger version, where $I$ is not assumed to be contained 
in the interior of a fundamental domain of $f$, but only contained in a fundamental domain. This weaker hypothesis
requires extra technical conditions.

\begin{rema} Under the hypotheses of Theorem~\ref{t.jump}, any   conjugates $f^iGf^{-i}$ 
and $f^j G f^{-j}$, $i\neq j$, have disjoint supports and therefore commute. Actually, 
the group $<f,G>$  generated by $f$ and $G$ 
is isomorphic to 
  $$<G,f>=\left(\bigoplus_\ZZ G\right)\rtimes \ZZ$$ 
where the factor $\ZZ$ is generated by $f$ and acts on  $\left(\bigoplus_\ZZ G\right)$ by conjugacy 
 as a shift of the $G$ factors (see Lemma~\ref{l.semidirect}).
\end{rema}

\subsubsection{Elementary groups and fundamental systems}
We now  define a class of subgroups $G\subset Diff^1([0,1])$, 
called \textbf{\emph{elementary groups}}, whose topological dynamics implies that they are $C^1$-close to the identity: 
 \emph{every group $G'$ topologically conjugated to $G$ is $C^1$-close to the identity}.

\begin{defi}A collection $(f_n)_{n\in\NN}$, $f_n\in Diff^1(\RR)$ is called a \textbf{\emph{fundamental system}} 
 if for every $n\in \NN$, $\RR\setminus Fix (f_n)$ consists in a
 unique connected component with compact closure $S_n$ (called the support of $f_n$)
and for every $n\in \NN$ there is a fundamental domain $I_n$ of $f_n$ such  that, 
for every $i,j\in \NN$  we have the following property:
\begin{itemize}
 \item either $S_i\subset I_j$
 \item or $S_j\subset I_i$
 \item or else $f_i$ and $f_i$ have supports with disjoint interiors
  $$\mathring{S_i}\cap\mathring{S_j}=\varnothing$$
 \end{itemize}
A group $G\subset Diff^1([0,1])$ is said to be an \textbf{\emph{elementary group}} if it is generated by a fundamental 
system supported in $[0,1]$. 
\end{defi}

\begin{rema}
 Every group $G'\subset Diff^1_+([0,1])$  topologically conjugated to 
an elementary group is an elementary group.

\end{rema}
Our main result is: 

\begin{theo}\label{t.realisation} Every elementary group $G\subset Diff^1_+([0,1])$ is $C^1$-close 
to the identity.
\end{theo}

%Theorem~\ref{t.realisation} give an answer to the following question~: 
%\centerline{\emph{What are the possible configurations for the intervals 
%of successive fixed points?}} 

Theorem~\ref{t.realisation} is obtained as a consequence of Theorem~\ref{t.jump2}
(the stronger version of Theorem~\ref{t.jump} above) and  Theorem~\ref{t.union}: by Theorem~\ref{t.union}
one only needs to consider groups generated by a finite fundamental system, and for these groups, 
Theorem~\ref{t.jump2} enables us to arg\"ue by induction on the cardinal of the fundamental system. 
\\
The groups contained in an elementary group are very specific. In particular, we can prove :

\begin{prop}\label{p.solvable2} Every finitely generated group contained in an elementary group is solvable. 
\end{prop}

\subsection{Examples and counter examples}

A simple solvable (non nilpotent) group is the Baumslag-Solitar group whose presentation is 
$$B(1,n)=<a,b| aba^{-1}=b^n>,$$
where $n$ is an integer 
such that $|n|>1$.  This group has a very natural affine action on $\RR$ 
($a$ acts as the homothety of ratio $n$ and $b$ is a translation) and on the circle (by identifying the affine group to
a subgroup of $PSL(2,\RR)$), and therefore on the segment $[0,1]$ 
(by opening the fixed point $\infty$ of the circle). In particular, $B(1,n)$ has analytic actions on $[0,1]$. These analytic actions
have been classified in \cite{BW}. All these actions have linked pairs of successive fixed points, and therefore are not $C^1$-close to the identity.
On the other hand, $B(1,n)$ admits $C^0$-actions without linked fixed points, but \cite{BMNR} shows that these $C^0$-actions cannot be $C^1$. 
\\
Indeed, \cite{Mc} shows the following Theorem, which holds  in any dimension:

\begin{theo}[A. McCarthy] For every $n$, $|n|>1$, and any compact manifold $M$, there is a $C^1$-neighborhood $U_n$ of the 
identity in $Diff^1(M)$ such that every morphism $\rho\colon B(1,n)\to Diff^1(M)$ with 
$\rho(a)\in U_n$ and $\rho(b)\in U_n$ satisfies $\rho(b)=id$.  
\end{theo}

In particular,  no subgroup of $\cC^1_{id}$ is isomorphic to $B(1,n)$. Actually, \cite{Mc} considers a more 
general class of groups called \emph{abelian by cyclic}: 

$$G_A=<a,b_1,\dots,b_k| b_ib_j=b_jb_i,\forall i,j, \mbox{ and } fg_if^{-1}=g_1^{a_{i,1}}\dots g_k^{a_{i,k}},  \forall i>,$$
where $A=(a_{i,j}) \in GL(k,\RR)$ has integer entries. \cite{Mc} shows that, 
if $A$ has no eigenvalue of modulus $1$, then there is no faithful action on compact 
manifolds such that the generators belong to a neighborhood $U_A$ of the identity in $Diff^1(M)$. In particular, no subgroup
in $\cC^1_{id}$ is isomorphic to $G_A$.
Actually, \cite{BMNR} showed recently that, for every faithfull $C^1$-action of the group $G_A$ on $[0,1]$, 
the diffeomorphism associated to $a$ admits a hyperbolic fixed point 
whose derivative is an eigenvalue of the matrix $A$. Consequently, no group in $\cC^1_{nonhyp}$ is isomorphic to $G_A$. 

%In a work in progress, we show that one may remove the commutativity conditions in the group $G_A$: consider the group 
%$\tilde G_A=<f,g_1,\dots,g_k|\forall i, fg_if^{-1}=g_1^{a_{i,1}}\dots g_k^{a_{i,k}} >,$ where $A=(a_{i,j})$ has no 
%eigenvalue of modulus equal to $1$. We announce that, for any morphism $\varphi\colon \tilde G_A\to Diff^1_+([0,1])$ such that   
%$\varphi(\tilde G_A)\in \cC^1_{id}$  then  $\varphi(g_i)=id$, for all $i$. The idea is that diffeomorphisms 
%$C^1$-close to the identity almost commute. As the group $\varphi(\tilde G_A)$ is $C^1$-close to the identity, its generators $\varphi(g_i)$ can
%be considered as arbitrarily $C^1$-close to the identity; hence  we recover asymptotically the commutativity.

This shows that there are algebraic obstructions for a group admitting faithful actions on $[0,1]$ to be $C^1$-close to identity. 
It is natural to ask if $\cC^1_{id}$ contains free groups.  The answer is positive as stated below:

\begin{theo}\label{t.free} There is a  group $G\subset Diff^1([0,1])$, $C^1$-close to  the identity 
(and totally rational),  such that $G$ is 
isomorphic to the free group $\FF_2$. 
\end{theo}

The proof of Theorem~\ref{t.free} consists in building a group $G_\omega=<f_\omega,g_\omega>\subset Diff^1_+([0,1])$, 
for all the reduced words $\omega$,
so that $G_\omega$ is $C^1$ close to the identity and the pair $\{f_\omega,g_\omega\}$ does not satify the 
relation $\omega$. Then, we glue together the groups $G_\omega$ in a way 
so that their supports are pairwize disjoint. Theorem~\ref{t.jump} is our main tool for building the groups 
$G_\omega$. 

\vskip 1cm

\noindent \textbf{Aknowlegment:} We thank Andres  Navas for numerous discussions and advises which simplified many arguments. 
The first author thanks Sylvain Crovisier and Amie Wilkinson with who  this story started in 2003, when we proved (but never wrote)
Theorem~\ref{t.linked} and \ref{t.translation}.

\section{Isotopies versus sequences of conjugacy to identity. }\label{s.cohomologique}

\subsection{The cohomological equation and the proof of Theorem~\ref{t.cohomologique} }
The aim of the section is the proof (suggested by A. Navas) of Theorem~\ref{t.cohomologique}.  Let us first start with the following observation:

\begin{lemm}
Consider $f\in Diff^1_+([0,1])$ and a sequence $\{h_n\}_{n\in\NN}$ with $h_n\in Diff^1_+([0,1])$. 
Then 

$$\left(h_nf h_n^{-1}\xrightarrow{C^1} Id\right)\Leftrightarrow 
\left(\log D h_n(f(x))-\log Dh_n(x)\xrightarrow{unif} -\log Df(x)\right)$$
 \end{lemm}
 \begin{demo} Just notice that the right term means that $\log D(h_nfh_n^{-1} (h(x)))$ converges uniformly to $0$, that is,
 $D (h_n f h_n^{-1})$ converges uniformly to $1$.  This implies that $h_n f h_n^{-1}$ is $C^1$-close to an isometry of $[0,1]$, 
 that is, to the identity map. 
 \end{demo}
 
A straightforward calculation implies: 
 \begin{coro}
  Assume that  $\psi_n\colon [0,1]\to \RR $ is a sequence  of continuous maps statisfying 
  $$\psi_n(f(x))-\psi_n(x)\xrightarrow{unif} -\log Df(x).$$
 Then $h_n\colon [0,1]\to \RR$ defined by 
 $$ h_n(x)=\frac{\int_0^x e^{\psi_n(t)} dt}{\int_0^1 e^{\psi_n(t)} dt}$$
 is a sequence of diffeomorphisms of $[0,1]$ such that  $h_nf h_n^{-1}\xrightarrow{C^1} Id.$
 \end{coro}
 
 For any continuous map $\psi\colon [0,1]\to \RR$, we will denote by $h_\psi\in Diff^1_+([0,1])$ the 
 diffeomorphism defined as above, that is: 
 $$ h_\psi(x)=\frac{\int_0^x e^{\psi(t)} dt}{\int_0^1 e^{\psi(t)} dt}.$$
Notice that:
 
 \begin{itemize}
 \item $\psi\mapsto h_\psi$ is a continuous map from $\cC^0([0,1],\RR)$ to $Diff^1_+([0,1])$. 
  \item $h_{\log Dg}=g$ for every $g\in Diff^1_+([0,1])$.
 \end{itemize}
Consequently, finding  diffeomorphisms $h_n$ conjugating  $f$ $C^1$-close to the identity is equivalent 
  to find almost solutions of the cohomological equation. The advantage of this approach  is 
  that the cohomological equation is a linear equation. As a consequence, convex sums of almost solutions are still almost solutions.

 \begin{lemm}\label{l.convexsum}
  Assume that  $\psi_n\colon [0,1]\to \RR $ is a sequence  of continuous maps statisfying 
  $$\psi_n(f(x))-\psi_n(x)\xrightarrow{unif} -\log Df(x).$$

Let $\psi_t, t\in[0,+\infty)$ be defined as follows:
  \begin{itemize}
  \item if $t=n\in \NN$, then $\psi_t=\psi_n$;
   \item if $t\in(n,n+1)$, then $\psi_t(x)= (n+1-t)\psi_n(x)+ (t-n)\psi_{n+1}(x)$.
  \end{itemize}
Then $\{h_t=h_{\psi_t}\}_{t\in [0,+\infty)}$ is a continuous path of diffeomorphisms statisfying 
  
  $$h_tf h_t^{-1}\limite{t\to +\infty}{C^1} id.$$ 
  
 \end{lemm}
 
\begin{demo}[Proof of Theorem~\ref{t.cohomologique}]
 Let $G$ be a group and assume that $h_n$ is a sequence of diffeomorphisms  such that 
 $h_n g h_n^{-1} \xrightarrow{C^1} id$ for every $g\in G$.  Let $\psi_n$ denote $\log Dh$ and 
  define $\psi_t, t\in[0,+\infty)$,  as in Lemma~\ref{l.convexsum} as convex sums of the $\psi_n$. 
  One denotes $h_t=h_{\psi_t}$.  
  Notice that, for $t=n\in \NN$, one has $h_t=h_n$, so that the notation is coherent. 
Now, for every $g\in G$, one has:
 $$h_tg h_t^{-1}\limite{t\to +\infty}{C^1} id .$$ 

\end{demo}

\subsection{Increasing union of groups $C^1$-close to the identity:proof of Theorem~\ref{t.union} }

\begin{demo}[Proof of Theorem~\ref{t.union}]
Let $\{G_n\}_{n\in \NN}$ be an increasing sequence of finitely generated groups $C^1$ close to the identity. 
For every $n\in\N$, let $k_n$ be an integer such that $\cS_n=\{g^n_1,\dots,g^n_{k_n}\}$ is a system of generators of $G_n$. 
\\
We denote $\cE_n=\bigcup_{i\leq n} \cS_i$.  
As $G_n$ is $C^1$-close to the identity, there is a sequence $h_{n,i}\in Diff^1_+([0,1])$ such that 

$$\forall g\in G_n, \quad  h_{n,i} gh_{n,i}^{-1}\overset{C^1}{\underset{i\to \infty}{\longrightarrow}} id. $$ 
Fix a sequence $\varepsilon_n>0$ such that $\varepsilon_n \underset{n\to \infty}{\longrightarrow} 0$.  For every $n$,
there is $i(n)$ so that
$$\forall g\in \cE_n,\quad \| h_{n,i(n)} gh_{n,i(n)}^{-1}-id\|_1<\varepsilon_n.$$
\\
As a straightforward consequence, for every $g\in G=\bigcup_{n\in \NN} G_n$ one has 
$$h_{n,i(n)} gh_{n,i(n)}^{-1}\overset{C^1}{\underset{n\to \infty}{\longrightarrow}} id.$$
According to Theorem~\ref{t.cohomologique}, this is equivalent to the fact that $G$ is $C^1$-close to the identity, ending  the proof. 
\end{demo}

As any countable group is an increasing union of finitely generated groups, we easily deduce: 
\begin{coro} If $\{G_n\}_{n\in\NN}$ is an increasing sequence of countable groups and if $G_n\in\cC^1_{id}$ for all $n$, then
$$G=\bigcup_{n\in \NN} G_n \in\cC^1_{id}.$$
 
\end{coro}

\section{Relative translation numbers for groups $C^1$-close to identity}\label{s.structure}
The aim of this section is to give a direct proof of Theorem~\ref{t.structure} in the case of groups of diffeomorphisms $C^1$-close to identity.   
This proof was essentially done (and never 
written) by the  first author with S. Crovisier and A.  Wilkinson in 2003. 
Section~\ref{s.crossing} presents the (later) proof due to A.  Navas, with completely different arguments,  for groups 
of homeomorphisms without linked fixed points.

Let us restate Theorem~\ref{t.structure} in the setting of groups $G\in \cC^1_{id}$: 

\begin{theo}\label{t.translation} Let $G\subset Diff^1_+([0,1])$ be a subgroup $C^1$-close to identity, $f$ in $G$, and $I$ a connected component 
of $[0,1]\setminus Fix (f)$. Assume $(f(x)-x)>0$ for $x\in I$. 
Then :
\begin{itemize}
 \item for every $g\in G$, either $g(I)=I$ or $g(I)\cap I=\emptyset$. 
 \item Let  denote $G_{I}=\{g\in G|\; g(I)=I\}$ the stabilizer of $I$. There is a unique 
 group morphism $\tau_{f,I}\colon G_{I}\to \RR$ with the following properties
 \begin{itemize}
  \item the kernel $Ker(\tau_{f,I})$ is precisely the set of elements $g\in G$ having a fixed point in $I$. 
  $$\left( g\in G_I \mbox{ and } \tau_{f,I}(g)=0\right)\Longleftrightarrow Fix(g)\cap I \neq \emptyset$$
  \item  $\tau_{f,I}$ is increasing:   given any $g,h\in G_{I}$, 
  assume there is
  $$\exists x\in I , g(x)\geq h(x) \Longrightarrow \tau_{f,I}(g)\geq \tau_{f,I}(h).$$ 
  \item $\tau_{f,I}(f)=1$. 
 \end{itemize}

\end{itemize}
 
\end{theo}

Theorem~\ref{t.translation} is the aim of this  Section~\ref{s.structure}.

\subsection{Backgroung on diffeomorphisms $C^1$-close to identity and compositions}

According to \cite[Lemme 4.3.B-1]{Bo} one has

\begin{lemm}\label{l.C1closetoid}
Let $M$ be a  compact manifold. For any $\varepsilon>0$ and $n\in\NN$, there is 
$\delta>0$ so that, for any $f\in Diff^1(M)$  whose $C^1$-distance to identity is 
less than $\delta$,  for every $x\in M$, for any $y$ with $\|y-x\|<n\|f(x)-x\|$,
one has:

$$\|(f(y)-y)-(f(x)-x)\|<\varepsilon  \|f(x)-x\|.$$
 
\end{lemm}

As a consequence \cite{Bo} shows: 

\begin{lemm}\cite[Lemme 4.3.D-1]{Bo}\label{l.C1composition} Let $M$ be a compact manifold.  For any $\varepsilon>0$ and $n\in\NN$, there is 
$\delta>0$ so that, for any $f_1, \dots f_n\in Diff^1(M)$  whose $C^1$-distance to identity is 
less than $\delta$,   for every $x\in M$, one has:

$$\|(f_n\dots f_1(x)-x)-\sum_i (f_i(x)-x)\|<\varepsilon \sup_i \|f_i(x)-x\|.$$

\end{lemm}

\subsection{Translation numbers}

In this section,  $G\subset Diff^1_+([0,1])$ denotes a  group  $C^1$-close to the identity. 
Thus, there is a $C^1$-continuous path $h_t\in Diff[0,1]$, $t\in[0,1)$,    so that 
$h_tgh_t^{-1}$ $C^1$-tends to $id$  as $t\to 0$, for every $g\in G$. 
For every element $g$ of $G$, we will denote $ g_t=h_tfh_t^{-1}$.  For every point
$x\in[0,1]$ we denote $x_t=h_t(x)$. 

As a direct consequence of Lemma~\ref{l.C1composition} one gets:

\begin{lemm}\label{l.iterates} Let $g\in G$ and $x\in [0,1]$ so that $g(x)\neq x$. 
Then, for every $n\in \ZZ$, one has:
$$\lim_{t\to 1}\frac{g_t^n(x_t)-x_t}{g(x_t)-x_t}=n$$
\end{lemm}

One deduces 

\begin{coro}\label{c.encadrement1} Consider $f,g\in G$ and $x\in[0,1]$ so that $f(x)> x$. Assume that there are $n>0$ and 
$m\in \ZZ$  
so that $$g^n(x)\in [f^m(x),f^{m+1}(x)].$$

Then, for every $t$ close enough to $1$ one has: 

$$\frac{g_t(x_t)-x_t}{f(x_t)-x_t}\in \left[\frac {m-1}n,\frac{m+2}n\right].$$

\end{coro}
Analoguous statements hold in the case $f(x)<x$ or $n<0$.

\begin{demo} The statement is immediately satisfied if $g(x)=x$.  We assume now $g(x)\neq x$. 

The conjugacy by $h_t$ preserves the order.  Thus, by assumption, one has 
$$\frac{g^n_t(x_t)-x_t}{f^m(x_t)-x_t}\in \left[1,\frac{f^{m+1}_t(x_t)-x_t}{f^m(x_t)-x_t}\right].$$ 

For every $t$ we denote $\alpha_t$ so that   $ \frac{g^n_t(x_t)-x_t}{f^m(x_t)-x_t} = \alpha_t\frac nm\frac{g_t(x_t)-x_t}{f(x_t)-x_t}$.  
Lemma~\ref{l.iterates} implies that $\alpha_t$ tends to $1$ as $t\to 1$.

Thus $$\alpha_t\frac{g_t(x_t)-x_t}{f(x_t)-x_t} \in \left[\frac mn, \frac mn\frac{f^{m+1}_t(x_t)-x_t}{f^m(x_t)-x_t}\right].$$
 
For $t\to 1$, Lemma~\ref{l.iterates} implies that $\frac mn\frac{f^{m+1}_t(x_t)-x_t}{f^m(x_t)-x_t}$ tends to $\frac{m+1}n$. 
One concludes by choosing $t$ large enough so that 
$[\alpha_t^{-1}\frac mn,\alpha_t^{-1}\frac{m+1}n]\subset \left(\frac {m-1}n,\frac{m+2}n\right)$.

\end{demo}

And conversely:
\begin{coro}\label{c.encadrement2} Consider $f,g\in G$ and $x\in[0,1]$ so that $f(x)> x$ 
and assume there  are $n, m\in \ZZ$, $n> 0$,  and a sequence $t_i\to 1$, $i\in\NN$ so that  
$$\frac{g_{t_i}(x_{t_i})-x_{t_i}}{f(x_{t_i})-x_{t_i}} \in \left[\frac mn,\frac{m+1}n\right]$$.  

Then  $g^n(x)\in [f^{m-1}(x),f^{m+2}(x)]$
\end{coro}

Analoguous statements hold in the case $f(x)<x$ or $n<0$.
The proof of Corollary~\ref{c.encadrement2} follows from the same estimates as 
Corollary~\ref{c.encadrement1} and is left to the reader. 

One deduces

\begin{coro}\label{c.translationponctelle}
Consider $f\in G$ and $x$ so that $f(x)\neq x$.  Then
\begin{enumerate}
 \item For any $g\in G$ the ratio $\frac{g_t(x_t)-x_t}{f(x_t)-x_t}$ as a limit 
 $\tau_f(g,x)\in \RR\cup\{-\infty,+\infty\}$ as 
 $t\to 1$.
 \item $\tau_{f^{-1}}(g,x)=\tau_f(g^{-1},x)=-\tau_f(g,x)$ with the convention $-\pm\infty=\mp\infty$. 
 
 \item $\tau_f(g,x)\in\{-\infty,+\infty\}$ if and only if $f$ has a fixed point in $(x,g(x))$ or $(g(x),x)$ 
 (according to the sign of $g(x)-x$). 
 
 \item let denote $G_{x,f}=\{g\in G, \tau_f(g,x)\in\RR\}$.  Then 
 $G_{x,f}$ is a subgroup of $G$ containing $f$ and $\tau_f(g,x)\colon G_{x,f}\to \RR$ is a morphism of 
 groups sending $f$ to $1$.
 
 \item \begin{itemize}
        \item if $\tau_f(g,x)\in\RR^*$ then $\tau_g(f,x)=\frac 1{\tau_f(g,x)}$
        \item $\tau_f(g,x)\in\{-\infty, +\infty\}\Longleftrightarrow\tau_g(f,x)=0$
       \end{itemize}
 \item if  $\tau_f(g,x)\in\RR^*$, then for every $h\in G$ one has
 \begin{itemize}
  \item $\tau_f(h,x)\in\{-\infty, +\infty\}\Longleftrightarrow\tau_g(h,x)\in\{-\infty, +\infty\}$
  \item $\tau_f(h,x)\in\RR \Rightarrow\tau_g(h,x)=\tau_g(f,x)\tau_f(h,x)$.
 \end{itemize}
\item $\tau_f(g,x)=0 \Longleftrightarrow g \mbox{ has a fixed point in } [x,f(x)]$.
\end{enumerate}

\end{coro}
\begin{demo}First notice that the sign of $\tau_f(g,x)$ is determined by the sign of $(f(x)-x)(g(x)-x)$. 

For the first item,  assume that there is a sequence $t_i$ so that 
$\frac{g_{t_i}(x_{t_i})-x_{t_i}}{f(x_{t_i})-x_{t_i}}$ is bounded, and hence, up to consider a subsequence, converges to 
some $\tau\in \RR$.  Then Corollary~\ref{c.encadrement2} implies that any rational estimate of $\tau$:
$$\frac mn <\tau< \frac {m+1}n$$ 
leads to a dynamical estimate of $g^n(x)$.  Now Corollary~\ref{c.encadrement1} implies that  
$\frac{g_t(x_t)-x_t}{f(x_t)-x_t}$ belongs to $[\frac{m-2}n,\frac{m+3}n]$ for any $t$ close enough to $1$. 

By considering finer rational estimates, one easily deduces that 
$\frac{g_t(x_t)-x_t}{f(x_t)-x_t}$ converges to $\tau$.  This concludes the proof of item 1. 

Item 2) is a direct consequence of Lemma~\ref{l.iterates}. 
Item 3) is a direct consequence of Corollaries~\ref{c.encadrement1} and ~\ref{c.encadrement2}.  
Assuming $g(x)>x$ and $f(x)>x$,  notice that if $f$ has no fixed points in $[x,g(x)]$ 
then there is $n>0$ with $f^n(x)>g(x)$ and therefore Corollary~\ref{c.encadrement1}
$\frac{g_t(x_t)-x_t}{f(x_t)-x_t}$ is bounded by $n+1$ for any
$t$ close to $1$. 

Item 4) is a direct consequence of Lemma~\ref{l.C1composition}.  Items 5) and 6) are easy 
consequences of the definition. Item 7) is a direct consequence of items 3) and 5).
\end{demo}

\begin{lemm}\label{l.translation}
Consider $f\in G$ and $x\notin Fix (f)$.  Let $I$ be the connected component of 
$[0,1]\setminus Fix(f)$ containing $x$. Then,
\begin{itemize}
 \item for every $g\in G$ and any $y\in I$ one has
$$\tau_f(g,x)=\tau_f(g,y).$$ 
We will denote it $\tau_{f,I}(g)$
\item $\tau_{f,I}(g)\in\RR\Longleftrightarrow g(I)=I$, that is, $g$ belongs to the stabilizer
$G_I$ of $I$. 
\item $\tau_{f,I}\colon G_I\to \RR$ is a morphism of group sending $f$ to $1$.
\item If $\tau_{f,I}(g)\in\{-\infty,+\infty\} $ then $g(I)$ and $I$  are disjoint.  
\end{itemize}

\end{lemm}
\begin{demo}
If $\tau_f(g,x)=0$, then Corollary~\ref{c.translationponctelle} asserts that $g$
as a fixed point in the interior of $I$.  One easily deduces that $\tau_f(g,y)=0$ for every $y\in I$. 
Therefore, Corollary~\ref{c.translationponctelle} implies that $[y,f(y))\cap Fix(g)\neq \emptyset$ for every
$y\in I$.  This implies in particular that the end points of $I$ are fixed points of $g$, so that $g\in G_I$.

Assume now that $\tau_f(g,x)\in \RR^*$.  Thus $g$ has no fixed point in $I$.  Let $J$ be the connected 
component of $x$ in $[0,1]\setminus Fix (g)$.  We have seen $\tau_g(f,x)=\frac1{\tau_f(g,x)}\in \RR^*$,
so that $f$ has no fixed points in
$J$.  One concludes $I=J$, implying that $g\in G_I$.  

Futhermore, Lemma~\ref{l.C1closetoid} implies that, for $t$ close enough to $1$, 
$g_t(y_t)-y_t$ and $f_t(y_t)-x_t$ are almost constant
on several fundamental domains (of $g_t$ and $f_t$, respectively), the error term being small compared 
with $g_t(x_t)-x_t$ and $f(x_t)-x_t$ respectively. 
As, by hypothesis, $g_t(x_t)-x_t $  and $f_t(x_t)-x_t$ remains in bounded ratio as $t\to 1$, 
one gets that the error 
term is small with respect to both. As a consequence, one gets that 
$\tau_f(g,y)=\lim_{t\to 1} \frac{g_t(y_t)-y_t}{f_t(y_t)-y_t}$ is locally constant, 
hence is constant on $I$. 
 
 Finally assume that $\tau_f(g,x)\in\{-\infty,+\infty\}$. This is equivalent to $\tau_g(f,x)=0$.  
 Let $J$ denote the connected component of $x$ in $[0,1]\setminus Fix (g)$.  Then we have seen that 
 each fundamental domain of $g$ in $J$ 
 contains  a fixed point of $f$. Notice that the extremities of $J$ are disjoint from $I$, 
 when at least one of the extremities of $I$ 
 is contained in $J$. 
 One easily concludes that $I\subset J$. 
 
 Now let denote $(a,b)=I$. Up to replace $g$ by $g^{-1}$, let us assume $g(x)-x>0$ so that $g(y)-y>0$ for every $y\in J$. 
 Now for every small $\varepsilon>0$, $[a+\varepsilon,g(a+\varepsilon)]$ contains a fixed point of $f$; 
 as there are no fixed points of $f$ between $a$ and $b$ one deduces  $g(a+\varepsilon)\geq b$. 
 As a consequence, $g(I)\cap I=\emptyset$, ending the proof.  
 
 \end{demo}

We conclude the construction of the relative translation number $\tau_{f,I}$ , 
and the proof of Theorem~\ref{t.translation}, 
by showing 
\begin{lemm}
 Let $f\in G$ and $x$ so that $f(x)>x$.  Let $I$ be the component of $x$ in $[0,1]\setminus Fix (f)$.  
 Then, for $g,g'$ in the stabilizer $G_I$, one has: 
 $$\exists y\in I g(y)\geq g'(y)\Rightarrow \tau_{f,I}(g)\geq \tau_{f,I}(g')$$
\end{lemm}
\begin{demo} Just notice than $\frac{g_t(y_t)-y_t}{f_t(y_t)-y_t}\geq \frac{g'_t(y_t)-y_t}{f_t(y_t)-y_t}$ for every $t$. 
\end{demo}

\subsection{Groups of diffeomorphisms without linked fixed points}

\begin{rema}   A group $G\subset Diff^1_+([0,1])$  is 
without linked fixed points if and only if,
given any $f,g\in G$, and given $]a,b[$ and $]c,d[$ be connected components of $[0,1]\setminus Fix(f)$ and 
 $[0,1]\setminus  Fix(g)$, respectively,   then, 
$$\mbox{ either }]a,b[\cap ]c,d[=\emptyset
\quad \mbox{ or }]a,b[= ]c,d[
\quad\mbox{ or }[a,b]\subset ]c,d[
\quad \mbox{ or }]a,b[\supset [c,d]$$
\end{rema}

Next straightforward lemma gives another formulation of  being without linked fixed points: 

\begin{lemm}\label{l.defilinked} A group $G\subset Diff^1_+([0,1])$  is without linked fixed points if and only if
given  $\{a,b\}$ and $\{c,d\}$ be successive fixed points of $f\in G$ and $g\in G$, respectively,  then 
$$\mbox{ either } \{c,d\}\cap ]a,b[=\emptyset
\quad \mbox{ or }[c,d]\subset ]a,b[
$$
\end{lemm}

\begin{theo}\label{t.enlaces} Let $G\subset Diff^1_+([0,1])$ be a subgroup $C^1$-close to identity. 
Then $G$ is without linked fixed points.
\end{theo}

\begin{demo} Let $f,g\in G$, and $I$ and $J$ be connected components of $[0,1]\setminus Fix (f)$ and 
$[0,1]\setminus Fix (g)$ respectively. 
Assume that one end point of $J$ belongs to $I$.  This implies that $\tau_{f,I}(g)=0$ .  Therefore, 
according to Corollary~\ref{c.translationponctelle}, $g$ has fixed points 
in every fundamental domains of $f$ in $I$.  This implies that $J$ is contained in the interior
of $I$. 

If no end point of $I,J$ belongs to the other interval, 
this implies that either $I=J$ or $I\cap J=\emptyset$. 
\end{demo}

\section{Groups without linked fixed points}

\subsection{Groups without linked fixed points and without crossing}
  In \cite{Na}, Navas  defines the notion of  \emph{groups without crossing} as follows:
 \begin{defi} One says that a group $G\subset Homeo_+([0,1])$ of homeomorphisms is \textbf{\emph{without crossing}} if 
given any $f,g\in G$, one has the following property:

Let $(a,b)$  be a connected components of $[0,1]\setminus Fix(f)$  Then, $\{g(a),g(b)\}\cap  (a,b)=\emptyset$
 \end{defi}

These two notions are equivalent:
 \begin{lemm}
  A group $G\subset Homeo_+([0,1]$ is without crossing if and only if it has no linked fixed points.
 \end{lemm}
 
\begin{demo}
Assume first that $G$ admits crossing: thus there is $f\in G$ and  $a,b$ be successive fixed point of $f$,
and $g\in G$ that $g(b)\in (a,b)$ (or $g(a)\in (a,b)$ but this case is analogous).  
\begin{itemize}
\item First assume that $g$ has a fixed point in $[a,b)$. Let $c,d$ be the successive 
fixed points of $g$ so that 
$b\in(c,d)$. Then $b\in(c,d)$ but $[a,b]\nsubseteq (c,d)$ so that $G$ is not without linked fixed points. 
\item Assume now that $g$ has no fixed points in $[a,b)$.
Then $gfg^{-1}$  admits
$c=g(a)<a$ and $d=g(b)\in (a,b)$ as successive fixed point so that $\{a,b\}$ and$ \{c,d\}$
are linked pairs of succesive fixe points.  
\end{itemize}

Conversely assume that $\{a,b\}$ and $\{c,d\}$ are linked pairs of fixed points, of $f\in G$ and $g\in G$, 
respectively.  Up to reverse the orientation or to exchange the role of $f$ and $g$, one may assume that
$d\in (a,b)$ and $c\notin(a,b)$. Up to exchange $f$ with $f^{-1}$, one may assume $f(d)<d$. 

Therefore $c\leq a <f(d)<d$ so that $f(d)\in(c,d)$: $G$ admits a crossing. 
\end{demo}

\subsection{Dynamics of groups without crossing}\label{s.crossing}

In this section $G\subset Homeo_+([0,1])$ is a group without crossing (or equivalently, without 
linked fixed points).  We fix an element $f\in G$ and a connected component $I$ of $[0,1]\setminus Fix (f)$. 
As $G$ is without linked fixed point, for any $g\in G$ 
\begin{itemize}
 \item either $g(I)\cap I=\emptyset$,
 \item  or $g(I)=I$, that is, $g$ belongs to the stabilizer $G_I$ of $I$.
\end{itemize}

In this section, we consider the dynamics of $G_I$ in restriction to $I$. 

As $f$ has no fixed points (by assumption) on $I$, every orbit  $G_I(y)$, $y\in I$ meets a (compact)
fundamental domain 
$[x,f(x)]$ of $f$.  One easily deduces (using Zorn lemma) that
\begin{lemm}
The action of $G_I$  on $I$ admits minimal closed  sets. 
\end{lemm}

Furthermore, the following classical result explains what are the possibilities: 
\begin{lemm} Let $H\subset homeo_+(\RR)$ be a group and assume there if $h\in H$ without fixed point.  
Then $H$ admits minimal sets, the union of the minimal sets is closed, and  
\begin{itemize}
 \item either there is a unique minimal set which 
 is is either $\RR$ or the product $\cC\times \ZZ$ where $\cC$ is a Cantor set.
 \item or there is $h\in H$ without fixed points so that every minimal of $G$ is exactly one orbit of 
 $g$.
 \end{itemize}

\end{lemm}

%We will apply this elementary general fact on our group $G$ without crossing. 

Let denote by $U\subset I$ the union of the open sets $I\setminus Fix(g)$, for $g\in G$ with 
$Fix(g)\cap I\neq \emptyset$.  That is:
$$U=\bigcup_{\{g\in G, Fix(g)\cap I\neq \emptyset\}}I\setminus Fix (g).$$

$U$ is an open set as union of open sets. The no linked fixed points property implies that 
each connected component of $U$ is the union of an increasing sequence of connected components of 
$I\setminus Fix(g_n)$; as a consequence, each connected component of $U$ is contained in a fundamental domain $[x,f(x)]$, $x\in I$.

\begin{rema}If $g$ has a fixed point $x$ in $I$, then it has another fixed point in $(x,f(x)]$;  otherwise, the next fixed point of $g$ 
will be so that (x,y) contains $f(x)$ and thus $G$ has a crossing. 

Thus, $g$ has fixed points in every fundamental domain of $f$ in $I$.
\end{rema}

One deduces that 
\begin{lemm}\label{l.minimal}
 $\La=I\setminus U$ is a nonempty closed subset invariant by $G_I$.  Furthermore, 
 for every $g\in G_I$ with fixed points in $I$,  
 the restriction
 of $g$ to $\La$ is the identity map.
 $$Fix(g)\cap I\neq \emptyset\Longleftrightarrow g|_\La=id_\La.$$
\end{lemm}
\begin{demo} Let $J$ be a component of $U$. If $g$ is the identity on $J$, in particular the extremities of $J$ are fixed point of $g$. 
Assume now that 
$x\in J$ is not fixed for $g$ and consider the pair $\{a,b\}$ of successive fixed points of $g$ around $x$. By construction of $U$,  $(a,b)\subset J$. 
One deduces easily that the extremities of $J$ are fixed points of $g$. Thus $g$ is the identity map on the boundary $\partial U$. Finally, 
every point $x\in I\setminus Fix(g)$ is contained in a connected component of $U$, so that $g$ is the identity map on $\La$, as announced. 
\end{demo}

Let $G_I^0$ be the set of elements  $g\in G$  so that $Fix (g)\cap I\neq \emptyset$.  
One deduces that
\begin{lemm}$G_I^0$ is a normal subgroup of $G_I$. Futhermore, $\Ga=G_I/G_I^0$ induces a group of 
homeomorphisms of $\La$ whose action is free (the non trivial element have no fixed points).
\end{lemm}

As a consequence one gets

\begin{lemm}
 Assume that the action of $G_I$ on $I$ has a unique minimal $\cM$ set which is either $I$ or $\cC\times \ZZ$. 
 Then, there is an increasing continuous map from $\cM$ to $\RR$ which induces a 
 semi-conjugacy of the action of 
 $\Ga$ on $\cM$ with a dense group of translation of $\RR$.
 
 Otherwize, $\Ga$ is monogene, and its action on $\La$ is conjugated to a translation. 
\end{lemm}
\begin{demo} If the minimal set is $I$ itself, this implies that the elements of $G_I$ have no fixed points in $I$. 
Therefore, H\"older theorem implies that the action is conjugated to a dense translations group. 

Assume now that the action of $G_I$ on $I$ has a unique minimal set $\cM$ homeomorphic to $\cC\times \ZZ$.  
By colapsing each closure of connected 
component of the complement of the minimal on a point, one defines a projection of $I$ to an interval, which is surjective
on $\cM$.  The action passes to the quotient and defines a free minimal action on the quotient interval: 
this quotient action is therefore conjugated to a dense translations group. 

When a minimal set $\cM_0$ is the orbit of an element $g\in G$, every minimal set is an orbit of $g$; 
the union of the minimal sets is a closed subset $K$
on which the dynamics is generated by $g$. In this case, the minimal sets are precisely the orbits of $g$ 
which are invariant by the whole group. If the  orbit  of $x$ for  $g$ is not a minimal of the action, 
there is 
$h\in G_I$ so that $x\neq h(x)$ but $x$ and $h(x)$ are in 
the same fundamental domain for $g$ hence belong to
the same connected component of $I\setminus K$; this implies that $h$ leaves invariant this component, hence has
fixed points in $I$.  This shows that $x$ belongs to the open set $U$.  In other words, we have shown that $\La=K$. 
As  $g$ is a generator of the action on $K$, one deduces that $g$ is a generator of $\Ga$, ending the proof.
\end{demo}

As a consequence we proved Theorem~\ref{t.structure} which provides for groups of homeomorphisms without crossing (or without linked fixed points) 
the same dynamical description as for 
group of diffeomorphisms $C^1$-close to identity, given by Theorem~\ref{t.translation}. 

%Let us restate this result: 

%\begin{theo}\label{t.sanscroisement} Let $G\subset Homeo_+([0,1])$ be a subgroup without crossing, 
%$f$ in $G$ and $I$ be a connected component 
%of $[0,1]\setminus Fix (f)$. Assume $f(x)-x>0$ for $x\in I$. 
%Then :
%\begin{itemize}
% \item for every $g\in G$, either $g(I)=I$ or $g(I)\cap I=\emptyset$. 
% \item Let  denote $G_{I}=\{g\in G| g(I)=I\}$.
% Then there is a unique group morphism $\tau_{f,I}\colon G_{I}\to \RR$ with the following properties
% \begin{itemize}
%  \item the kernel $Ker(\tau_{f,I})$ is precisely the set $h\in G$ having a fixed point in $I$. 
%  $$\left( g\in G_I \mbox{ and } \tau_{f,I}(g)=0\right)\Longleftrightarrow Fix(g)\cap I \neq \emptyset$$
%  \item  $\tau_{f,I}$ is increasing in the following  meaning:   given any $g,h\in G_{I}$, assume there is
%  $x\in I$ so that $g(x)\geq h(x)$.  Then $\tau_{f,I}(g)\geq \tau_{f,I}(h)$. 
%  \item $\tau_{f,I}(f)=1$. 
% \end{itemize}
%\end{itemize}
%\end{theo}

\subsection{Countable familly of intervals}
\begin{prop}\label{p.countable}
 Let $G\subset Homeo_+([0,1])$ be a subgroup without crossing. 
 Then, the family of pairs of successive fixed points of the elements of $G$ is at most countable. 
\end{prop}
\begin{demo} Le $G$ be a group without crossing  and let $\cP_n$ be the set of pairs of succesive 
fixed points $\{a,b\}$ with $|a-b|\geq \frac 1n$. For proving the lemma it is enough to prove that $\cP_n$ is at most countable, for every 
$n\in \NN\setminus\{0\}$. 

Let $\{a,b\}\in\cP_n$ be such a pair and $g\in G$ so that $\{a,b\}$ are successive fixed points of $g$. Up to replace $g$ by $g^{-1}$, one assumes that
$g(x)-x>0$ for $x\in(a,b)$.

Let $\varepsilon>0$ so that $0<g(x)-x<\frac 1n$ for $x\in (a,a+\varepsilon)$ and denote $c=g^{-1}(a+\varepsilon)$. 

\begin{clai} Any pair $\{p,q\}\in\cP_n$ is disjoint from $(a,c]$.
\end{clai}
\begin{demo} Assume $\{p,q\}$ meets $(a,c]$. As $G$ is without crossing and as $\{a,b\}$ is a pair of succesive fixed points for $g$, the pair
$\{p,q\}$ is contained in a fundamental domain of $g$ that is in $[p, g(p)]$. However our choice of $c$ implies that $|p-q|<\frac 1n$ contradicting the fact 
that $\{p,q\}\in \cP_n$.
\end{demo}

Now, to every pair $\{a,b\}$, one can associated the largest open interval $J_{a,b}=(a,d)$ so that $(a,d)$ is disjoint from any pair $\{p,q\}\in\cP_n$.  
In other words, $d=\inf\{p>a,\exists q \mbox{ so that } \{p,q\}\in \cP_n\}$. 
The claim asserts that this open interval $J_{a,b}$ is not empty (that is, $d>a$).

By construction, if $\{a_1,b_1\},\{a_2,b_2\}$ are two distinct pairs in $\cP_n$, then $J_{a_1,b_1}\cap J_{a_2,b_2}=\emptyset$. 
Now, any family of disjoint open intervals is countable, concluding.

\end{demo}

%%%%%%%%%%%%%%%%%%%%%%%%%%%%%%%%%%%%%%%%%%%%%%%%%%%%%%%%%%%%%%%%%%%%%%%%%%%%%%%%%%%%%%%%%%%%%%%%%
\subsection{A characterization of crossing, and entropy}

In the next section we will show that groups of diffeomorphisms admiting a crossing have hyperbolic fixed points. 
The main step for the proof is the next lemma, which provides a dynamical characterization of the existence of crossings:

\begin{lemm}\label{l.entropy} Consider $G\subset Homeo_+([0,1])$.  Assume that $G$ admits crossings.  
Then there are $h_1,h_2\in G$ and a segment $I\subset [0,1]$ so that $h_1(I)$ and $h_2(I)$ are disjoint 
segments contained in $I$.

As a direct consequence, the topological entropy of the semi group generated by $h_1^{-1}$ and $h_2^{-1}$ is $\log 2$. 
 
\end{lemm}
\begin{demo}According to Lemma~\ref{l.defilinked}, there are $f,g\in G$ and successive fixed points 
$\{a,b\}$   and $\{c,d\}$ of $f, g$ respectively,  so that  
(up to reverse the orientation, and to exchange $f$ with $g$) $b\in (c,d)$ but $(a,b)\nsubseteq(c,d)$, 
that is 
$a<c<b<d$. 

Up to exchange $f$ with $f^{-1}$ and $g$ with $g^{-1}$, on may assume that $f(x)-x>0$ on 
$(a,b)$ and $g(x)-x>0$ on $(c,d)$.

Consider $x_0\in (a,c)$, and fix $I=[x_0,b]$.   Then for $n>0$ large $f^n(I)$ is a small segment in $I$ arbitrarily close to
$b$.  Then,  for positive $m$,  $g^{-m} f^n(I)$ form an infinite collection of disjoint segments contained  in $[c,b)$  
hence contained in the interior of $I$.
\end{demo}

\subsection{Group without hyperbolic fixed point}

Recall that a group $G\subset Diff^1_+([0,1])$ is \textbf{\emph{without hyperbolic fixed points }} if, for every $g\in G$, one has 
$Dg(x)=1$ for every $x\in Fix (g)$. The aim of this section is to prove Theorem~\ref{t.sanshyperbolique}, that is, 
\emph{if $G$ is without hyperbolic fixed points, then it is without crossing.}

We present here a proof of  A. Navas. Let us start by stating two lemmas.
\begin{lemm}\label{l.ell} Let $I$ be a segment and $f,g\colon I\to I$ be diffeomorphisms onto their images, and  so that 
$f(I)\cap g(I)=\emptyset$.  Then there is an infinite sequece $\omega_i$, $i\in\NN$, $\omega_i\in\{f,g\}$ 
so that  
$$\limsup \frac1n \log\ell(\omega_{n-1}\omega_{n-2}\cdots\omega_0(I))<0$$

where $\ell$ denotes the length.
\end{lemm}
The proof of Lemma~\ref{l.ell} is postponed at the end of the section. 
\begin{lemm}\label{l.distorsion}
Let $f,g\in Diff^1([0,1])$ be two $C^1$-diffeomorphisms and assume there is a segment $I$ for which there is an
an infinite word $\omega_i$, $i\in\NN$, $\omega_i\in\{f,g\}$ 
so that  
$$\limsup \frac1n \log\ell(\omega_{n-1}\omega_{n-2}\cdots\omega_0(I))<0$$

Then for every $t\in I$ one has 
$$\limsup \frac1n \log\left|D\left((\omega_{n-1}\omega_{n-2}\cdots\omega_0\right)(x)\right|=
\limsup \frac1n \log\ell(\omega_{n-1}\omega_{n-2}\cdots\omega_0(I))<0$$
 
\end{lemm}
\begin{demo}This lemma looks like a distorsion control,
which does not exists in the $C^1$-setting. However the proof
is a easy consequence of the uniform continuity of $Df$ and $Dg$: the length of 
$\omega_{n-1}\omega_{n-2}\cdots\omega_0(I)$  tends to $0$, by assumptions. 

Therefore  
$$\log |D (\omega_n) ( \omega_{n-1}\omega_{n-2}\cdots\omega_0 (x)|-\log \frac{\ell(\omega_{n}\omega_{n-1}
\cdots\omega_0(I))}{\ell(\omega_{n-1}\omega_{n-2}\cdots\omega_0(I))}
\xrightarrow{unif}0, \mbox{ for } n\to \infty$$

One conludes by noticing that 
$$\frac1n \log\left|D\left((\omega_{n-1}\omega_{n-2}\cdots\omega_0\right)(x)\right|=
\frac 1n \sum_0^{n-1} \log |(D\omega_i)(\omega_{i-1}\cdots \omega_0(x))|,$$
and

$$\frac1n \log\ell(\omega_{n-1}\omega_{n-2}\cdots\omega_0(I))= 
\frac 1n \left(\log \ell(I)+\sum_1^{n-1}\log\frac{\ell(\omega_{i}\cdots\omega_0(I))}
{\ell(\omega_{i-1}\cdots\omega_0(I))}\right)  $$

\end{demo}

Before giving the proof of Lemma~\ref{l.ell}, let us conlude the proof of Theorem~\ref{t.sanshyperbolique}. 

\begin{demo}[Proof of Theorem~\ref{t.sanshyperbolique} ]Let $G\subset Diff^1_+([0,1])$ be a group with a crossing.  According to 
Lemma~\ref{l.entropy} there are $f,g\in G$ and a segment $I\subset [0,1]$ so that $f(I)$ and $g(I)$ are disjoint segments 
contained in $I$.

According to Lemma~\ref{l.ell} there is an infinite word $\omega_i\in\{f,g\}$ so that the size  of
the segments $\omega_{n-1}\cdots \omega_0(I)$ decreases exponentially.  Then, Lemma~\ref{l.distorsion}  implies that
for every $x\in I$  one has $|D\omega_{n-1}\cdots \omega_0(x)|<1$.  As $\omega_{n-1}\cdots \omega_0(I)\subset I$
there is a fixed point in $I$ and this fixed point has derivative $<1$, so that it is hyperbolic. 

This implies that $G$ is not without hyperbolic fixed point, concluding the proof.
\end{demo}

\begin{demo}[Proof of Lemma~\ref{l.ell}]

For any $n\in \NN$ let $\Om_n=\{f,g\}^n$ be the set of words $(\omega_i)_{i\in\{0,\dots,n-1\}}$ of length $n$, 
with letters $\omega_i\in \{f,g\}$.  In particular the cardinal of $\Om_n$ is
$$\# \Om_n=2^n$$ 

As the intervals $\omega_{n-1}\dots\omega_0(I)$ are pairwize disjoint, one expects that the length is in 
general no much more than $\frac 1{2^n} \ell(I)$. 

Let denote 
$$B_n=\left\{(\omega_i)_{i\in\{0,\dots,n-1\}}\in \Om_n| \ell(\omega_{n-1}\dots\omega_0(I))\geq \ell(I)\cdot(\frac 23)^n\right\}.$$
A simple calculation shows
$$\# B_n\leq \left(\frac 32\right)^n$$

Thus $\frac{\# B_n}{\# \Om_n}\leq \left(\frac 34\right)^n$. 

Choose $0<\varepsilon<1$ and  $T>0$ so that $\sum_1^\infty \left(\frac 34\right)^{nT}<\varepsilon$. 

Let denote 
$$\cB^T_n=\left\{(\omega_i)_{i\in\{0,\dots,nT-1\}}\in \Om_{nT}| \exists i>0, (\omega_i)\in B_{iT}\right\}.$$ 
A simple calculation shows 
$$\frac{\# \cB^T_n}{\# \Om_{nT}}\leq \sum_1^{n-1}\left(\frac 34\right)^{Tn}.$$

Let $\Om_{\infty}=\{f,g\}^{\NN}$ be the set of infinite words in letters $f,g$.  Its a Cantor set. 
Consider 

$$\cG^T_n=\left\{(\omega_i)_{i\in\NN}\in\Om_\infty| (\omega_i)_{i\in\{0,\dots, nT\}}\notin \cB^T_n\right\}$$
 and  

$$\cG^T_\infty=\left\{(\omega_i)_{i\in\NN}\in\Om_\infty|\forall n>0,  (\omega_i)_{i\in\{0,\dots, nT\}}\notin \cB^T_n\right\}$$
 
Then $\cG^T_n$ is a decreasing sequence of compact subsets of $\Om_\infty$ and 
$$\cG^T_\infty=\bigcap_{n>0} \cG^T_n.$$

The fact that $ \frac{\# \cB^T_n}{\# \Om_{nT}}<1$ implies that $\cG^T_n$ is not empty.  One deduces that $\cG^T_\infty$ is not empty 
(as a decreasing sequence of non empty compact sets). 

One concludes by noticing that, for every word $(\omega_i)_{i\in\NN}\in G^T_\infty$ one has 
$$\limsup\frac 1n\log \ell(\omega_{n-1}\cdots\omega_0)(I)\leq \log(\frac 23).$$

\end{demo}

\begin{rema}The proof above give much more: if one endows $\Om_\infty$ with the measure whose weight on each 
cyclinder of length $n$ is $\frac 1{2^n}$, then for almost every word $(\omega_i)$ in $\Om_\infty$ the 
exponential rate of deacreasing of the length is upper bounded by $-\log 2$. 
\end{rema}

\section{Completion of  a group without crossing}
\subsection{Complete group witout crossing}

Consider a homeomorphism $g\in Homeo_+([0,1])$.  We say that a homeomorphism $h\in Homeo_+([0,1])$ 
is \textbf{\emph{induced}} by 
$g$ if for every $x\in [0,1]$, one has 
$$h(x)\in \{x,g(x)\}$$ 

In other words, $h$ is obtained from $g$ in replacing $g$ by the identity map in the  union of some connected components
of $[0,1]\setminus Fix(g)$. More precisely, one easily check:

\begin{lemm} Given $g\in Homeo_+([0,1])$,  a homeomorphism $h$ is induced by $g$ if
and only if there is  a family $\cI$ of 
connected components of $[0,1]\setminus Fix (g)$ so that  $h$ is   the map  $g_{\cI}$ defined as follows: 
\begin{itemize}
 \item if $x\in \bigcup_{I\in \cI} I$,  $g_{\cI}(x)=g(x)$
 \item otherwise $g_{\cI}(x)=x$.
\end{itemize}
\end{lemm}

\begin{defi}Let $G\subset Homeo_+([0,1])$ be a group without crossing. 
One says that $G$ is \textbf{\emph{complete}} if, for any
$g\in G$,  any homeomorphism induced by $g$ belongs to $G$. 
\end{defi}

The aim of this section is to show
\begin{prop}\label{p.completion} Any group $G$ without crossing is contained in a complete group without crossing.
\end{prop}

Let $G\subset Homeo_+([0,1])$ be a group without crossing.  We will denote by $I(G)\subset Homeo_+([0,1])$ 
the group generated by all the elements $h$ induced by elements $g\in G$. 

\begin{rema}For every $h\in I(G)$ and $x\in[0,1]$ there is $g\in G$ with $h(x)=g(x)$.  
In other words, $I(G)$ and $G$ have the same orbits. 
\end{rema}

\begin{lemm}\label{l.successive} Assume $G$ is without crossing.
If $\{a,b\}$ are successive fixed points for some $h\in I(G)$ then $\{a,b\}$ are successive 
fixed points for some $g\in G$.
\end{lemm}
\begin{demo}
Assume, by absurd, that there is a pair of successive fixed points $\{a,b\}$ for an element $h\in I(G)$, but which is not of successive 
fixed points for $G$. So, $h$ can be written as $h= h_n\cdots h_1$ where $h_i$ is induced by $g_i\in G$.  
We choose the pair $\{a,b\}$  so that $n$ is the smallest possible. Fix a point $x\in (a,b)$.

For every $i\in\{1,\dots,n\}$, let $a_i$ be the largest  fixed points of $g_i$ less than or equal to $x$,
and $b_i$ the smallest larger than or equal to $x$.  
 If $g_i(x)\neq x$, then $\{a_i,b_i\}$ are successive fixed points of $g_i$ (otherwise, $a_i=b_i=x$). 
As $h(x)\neq x$, there is at least one index  $i$ for which $g_i(x)\neq x$.

As $G$ is without crossing, the intervals $(a_i,b_i)$ are totally ordered for the inclusion. 
The union $I=\bigcup_1^n (a_i,b_i)$  is invariant by all the $g_i$:  indeed, either $I$ is an interval of successive
fixed points of $g_i$ or $g_i$ has a fixed point in $I$, an then belongs to the stabilizer of $I$. One deduces 
$I$ is fixed by  all the $h_i$.  As a consequence one gets $(a,b)\subset I$. 

Notice that there is $j\in\{1,\dots ,n\}$ so that $I=(a_j,b_j)$. 
Let $i_1,\dots, i_k$ be the set of indices so that $I=(a_{i_j},b_{i_j})$. If there is $i$ so that $h_i=id $ on $I$, 
then $n$ was not the minimum number. As the indices $i_j$ have be chosen so that
$I$ is a component of $[0,1]\setminus Fix(g_{i_j})$, this implies that 
$h_{i_j}=g_{i_j}$ on $I$.  Thus, one gets the same interval $\{a,b\}$ if we substitute the $ h_{i_j}$ by $g_{i_j}$.
Thus we will now assume $h_{i_j}=g_{i_j}$.

One considers now the group generated by the restriction of the $g_i$ to $I$. It is a group without crossing and
the restrictions of the $h_i$ to $I$ are induced by $g_i$.  We consider a minimal set $\cM\subset I$ of the action of the group generated by the 
$g_i$ on $I$:  every $g_i$, $i\notin\{i_1,\dots, i_k\}$, has fixed points in $I$; according to Lemma~\ref{l.minimal}, $g_i$
induces the identity map on $\cM$.  Therefore the same happens for $h_i$, $i\notin\{i_1,\dots, i_k\}$. 
So the action of $h$ on $\cM$ is the same as $g= g_n\cdots g_1$

Let us now consider the translation number $\tau$ relative to $g_{i_1}$ on $I$. 
Then $\tau(g)$ is the sum of the $\tau(g_{i_j})$. 

First assume that $\tau(g)\neq 0$.  Thus the orbits of $g$ on the minimal $\cM$ goes from one extremities 
of $I$ to the other, and so do the orbit of $h$ on $\cM$.  
In particular, $I=\{a,b\}$, so that $\{a,b\}$ is a pair of succesive fixed point of $g_{i_1}$,
contradicting the definition of 
$\{a,b\}$.

Thus $\tau(g)=0$. This implies $k>1$. Now we will use the fact that, if $f\in G$ and $h$ is induced by $g\in G$, 
then $fhf^{-1}$ is induced by $fgf^{-1}$ in $G$.   By using inductively the elementary fact 
$f(f^{-1}hf)= hf$, we can rewrite the word 

$$h_n\dots h_{n_k+1}g_{n_k}h_{n_k-1} \dots h_{n_1+1}g_{n_1}h_{n_1-1}\dots h_1= g_{n_k}\dots g_{n_1}\tilde h_{n-k}\cdots \tilde h_1$$
where the $\tilde h_i$ are induced by elements of $G$. However, $\tilde g=g_{n_k}\dots g_{n_1}$ 
belongs to $G$. Thus,
one can rewrite this word has $\tilde g \tilde h_{n-k}\cdots \tilde h_1$, which has only $n-k+1<n$ letters. 
This contradicts the fact that $n$ was chosen realizing the minimum. 
\end{demo}

\begin{coro}\label{c.translationcompletion}
Assume tha the group $G$ is without crossing. Then $I(G)$ is without crossing. 

Let $\{a,b\}$ be a pair of successive fixed points, for some $f\in G$. Then, the image of the translation 
number associated to $\{a,b\}$ is the same for $G$ and $I(G)$: 
$$\tau_{f,[a,b]}(G)=\tau_{f,[a,b]}(I(G))\subset \RR$$
\end{coro}
\begin{demo} Groups without crossing are  the groups without linked fixed points. This property is a 
property of the 
set of pairs of successive fixed points. According to Lemma~\ref{l.successive} the pairs of successive 
fixed points are the same for $G$ and for $I(G)$, concluding.

Let now $\{a,b\}$ be a pair of successive fixed points for  some $g\in G$.  
Consider the stabilizer $I(G)_{[a,b]}$. 

As $I(G)$ is without crossing, the translation number relative to $g$ extends on $I(G)_{(a,b)}$. 
Furthermore, the action of $G_{[a,b]}$ on $(a,b)$ admits a minimal set $\cM$. 
For every $x\in [0,1]$ and every $h\in I(G)$ there is $g\in G$ so that $h(x)=g(x)$.  
This implies that
\begin{clai}The minimal set $\cM$ is invariant under the action of $I(G)_{(a,b)}$
\end{clai}
\begin{demo} 
if $h\in I(G)$ and $x\in \cM$, and if $h(x)\in (a,b)$,  then there is $g\in G$ with $g(x)=h(x)$; then
$g(x)\in \cM$ (because $\cM$ is invariant by $G_{(a,b)}$; one concludes $h(x)\in\cM$. Thus the minimal set $\cM$ is invariant by $I(G)_{(a,b)}$.
\end{demo}

  Now the action of 
$I(G)_{(a,b)}/Ker (\tau_{g,(a,b))}$ is a  free action. For every $h\in I(G)$ and $x\in \cM$ there is $g\in G$
with $g(x)=h(x)$; this implies that $g$ and $h$ coincides on $\cM$, and thus 
$\tau_{f,(a,b)}(h)=\tau_{f,(a,b)}(g)$,  concluding.
\end{demo}

We don't know if, in general, $I(G)$ is a complete group, that is, if $I(I(G))=I(G)$.  
For this reason, let us denote $I^n(G)$ defined as $I^{n+1}(G)=I(I^n(G))$. The sequence $I^n(G)$ 
is an increasing sequence of groups.  We denote 
$$I^\infty(G)=\bigcup_{n\in\NN} I^n(G)$$

Next Lemma ends the proof of Proposition~\ref{p.completion}:
\begin{lemm}\label{l.completion} For every $G\subset Homeo_+([0,1])$ without crossing, $I^{\infty}(G)$ is a complete group without crossing. 
Furtheremore:
\begin{itemize}
\item the orbits of $I^{\infty}(G)$ and of $G$ are equal;
 \item any pair of successive fixed points $\{a,b\}$ of $I^{\infty}(G)$ are is a pair of successive fixed points of $G$;
 \item for any pair $\{a,b\}$ of successive fixed points of some $f\in G$, the images
 $\tau_{f,[a,b]}(G)$ and $\tau_{f,[a,b]}(I^{\infty}(G))$ are equal;
 \item any complete group without crossing containing $G$ contains $I^{\infty}(G)$. 
\end{itemize}

\end{lemm}
\begin{demo}The unique non-trivial point is that $I^{\infty}(G)$ is complete. For that, it is enough to show that, 
if $g\in I^{\infty}(G)$ then every homeomorphism $h$ induced from $g$ also belongs to $I^{\infty}(G)$.  
Notice that there is $n$ so that $g\in I^n(G)$;  thus $h\in I^{n+1}(G)$, concluding. 
 
\end{demo}

The group $I^{\infty}(G)$ is called \textbf{\emph{the completion}} of $G$.

\subsection{Completion of  groups $C^1$-close to identity or without hyperbolic fixed points.}

The aim of this section is to prove that the completion of groups $C^1$-close to identity or without hyperbolic fixed points
are $C^1$-close to identity or without fixed points, respectively. 
Notice that: 
\begin{rema}
 For any sequence $\cH=\{h_n, n\in\NN\}$ of diffeomorphisms of $[0,1]$, let $G_{\cH}$ be the set 
 of diffeomorphisms so that 
  $h_n g h_n^{-1}\xrightarrow{C^1} id$  as $n\to\infty$.  Then $G_{\cH}$ is a group $C^1$-close to identity. 
  \end{rema}
  
\begin{lemm}
  Consider  a group $G$,  $C^1$-close to the identity, and $h_n$ 
 a sequence of diffeomorphisms so that $h_n g h_n^{-1}\to id$ for every $g\in G$.  Consider 
an element $g\in G$, and  a family $\cI$ of 
connected components of $[0,1]\setminus Fix (g)$. 

Then, the induced map $g_{\cI}$ 
(equal to $g$ on the components in $\cI$ and equal to $id$ out of these components),
is a $C^1$ diffeomorphism of $[0,1]$.  Furthermore,  $h_n g_{\cI} h_n^{-1}\xrightarrow{C^1} id$  
as $n\to\infty$.  In other words, $g_{\cI}\in G_{\cH}$. 
\end{lemm} 
\begin{demo} First notice that $g$ has no hyperbolic fixed point: the derivative of $g$ is $1$ 
at each extremities
of the components of $\cI$. One deduces that $g_{\cI}$ is a diffeomorphism.

Then, $h_n g_{\cI} h_n^{-1}$ is induced by $h_n g h_n^{-1}$. Therefore its $C^1$-distance to identity is smaller that the distance from  $h_n g h_n^{-1}$ to identity.
\end{demo}

\begin{coro}\label{c.c1completion} Consider  a group $G$,  $C^1$-close to identity, and a sequence $\cH=\{h_n\}$ 
of diffeomorphisms so that $h_n g h_n^{-1}\to id$ for every $g\in G$. Then the completion 
$I^{\infty}(G)$ is contained in $G_{\cH}$.  In particular, $I^{\infty}(G)$ is $C^1$-close to identity.
\end{coro}

\begin{lemm}\label{l.completesanshyp} If $G\subset Diff^1_+([0,1])$ is a group without hyperbolic fixed points, 
then the completion $I^{\infty}(G)$ is  contained in $Diff^1_+([0,1])$ and is without hyperbolic 
fixed points.
\end{lemm}
\begin{demo}It is enough to show that $I(G)$ is a goup of diffeomorphsim witout hyperbolic fixed points. 

As seen before, as the elements $g\in G$ are diffeomorphism without hyperbolic fixed points, every induced
map is a diffeomorphism. It remains to show that $I(G)$ is without hyperbolic fixed points. 

Assume that $a$ is a hyperbolic fixed point of an element $h\in I(G)$. Then $a$ is an isolated fixed point of $h$. 
Let $b$ be the next fixed point, that is $\{a,b\}$ is a pair of succesive fixed points of $h$. 
According to Lemma~\ref{l.successive} there is $g\in G$ so that $\{a,b\}$ are successive fixed points of 
$g$.  Thus $\tau_{g,(a,b)}(h)$ is well defined and finite.  However,  the derivative $g'(a)$ is $1$ because 
$G$ is without hyperbolic fixed points.  One easily deduces that $h'(a)=1$ 
(otherwise,$\tau_{g,(a,b)}(h)$ is infinite),   contradicting the hypothesis.  
\end{demo}

\subsection{Algebraic presentation: specific subgroups}

The finitely generated groups $C^1$-close to identity may be complicated.  However the complete groups without 
crossing admits special subgroups with a simple presentation. 

An element in $G$ is called \textbf{\emph{simple}} if  $[0,1]\setminus Fix (g)$ 
consist in a unique interval whose  closure is  
the support of $g$ and denoted $supp(g)$. 

\begin{rema} 
\begin{enumerate}
\item Let $G\subset Homeo_+([0,1])$ be a complete group without crossing.  
Each $g\in G$ is the limit, for the $C^0$-topology, of the products of the 
induced simple elements $g_I$ where $I$ covers the set of connected components of $[0,1]\setminus Fix(g)$.
\item Let $G\subset Diff_+^1([0,1])$ be a complete group without hyperbolic fixed points. 
Each $g\in G$ is the limit, for the $C^1$-topology, of the products of the 
induced simple elements $g_I$ where $I$ covers the set of connected components of $[0,1]\setminus Fix(g)$.
\end{enumerate}
\end{rema}

Let  $G\subset Homeo_+([0,1])$ be a group  without crossing. According to Theorem~\ref{t.sanscroisement}, 
given any  pairs $f,g$ simple elements of  $G$, 
one has one of the following possibilities: 
\begin{itemize}
 \item  either the supports $supp(f)$ and $supp(g)$ have disjoint interiors
 \item or the supports are equal
 \item or else the support of one of the diffeomorphism is contained in a fundamental domain of the other.
\end{itemize}

The group generated by $f$ and $g$ depends essentially on these $3$ configurations and admits a simple presentation 
if $supp(f)\neq supp(g)$.

\begin{prop}\label{p.simplesubgroup} Let $G$ be group without crossing let $g_1$ and $g_2$ be two simple elements of $G$.
Then 
\begin{itemize}
 \item if $Int(supp(g_1)\cap supp(g_2))=\emptyset$ then $g_1$ and $g_2$ commute:
 $$<g_1,g_2>=\ZZ^2.$$
 
 \item if    $supp(g_2)\subset Int(supp(g_1))$,  
  the group $<g_1,g_2>$ admits as unique relations 
 the fact that the conjugates of $g_1^ig_2g_1^{-i}$, $i\in\ZZ$, pairwize commute. More precisely
 $$<g_1,g_2>=\left(\bigoplus_{\ZZ} \ZZ\right)\rtimes \ZZ$$
 where $\ZZ$ acts by conjugacy on $\left(\bigoplus_{\ZZ} \ZZ\right)$ as a shift of the $\ZZ$ factors. 
 \end{itemize} 
 \end{prop}
 \begin{proof} We just need to prove the second point.  If $supp(g)\subset Int(supp(f))$, 
 then the images by $f^i$ of $supp(g)$ are pairwize disjoint. This proves 
  that the $g_i= f^i g f^{-i}$, $i\in\ZZ$,  pairwize commute. This allows us to define a morphism
 $$\varphi\colon \left(\bigoplus_{\ZZ} \ZZ\right)\rtimes \ZZ= <a,b| [a^iba^{-i},a^jba^{-j}], i,j\in \ZZ\} \to 
 <f,g>,$$
 so that $\varphi(a) = f$ and $\varphi(b)= g$. 
 It remains to show that $\varphi$ is injective.  For that, notice that 
 every element of the group can be written as $a^ib_{i_1}^{\beta_1}\dots b_{i_k}^{\beta_k}$ with pairwize distinct $b_{i_k}=a^{i_k}ba^{-i_k}$.  
 The image is 
 $f^ig_{i_1}^{\beta_1}\dots g_{i_k}^{\beta_k}$.  The translation number relative to $f$ is $i$ so that the element is not the identity unless $i=0$. 
 In that case the element is $g_{i_1}^{\beta_1}\dots g_{i_k}^{\beta_k}$ which vanishes only if all the $\beta_i$ vanish, ending the proof.

\end{proof}

The proposition admits a straightforward generalization as follows: 
 
 \begin{lemm}\label{l.semidirect}
\begin{enumerate} 
\item  Let $I$, $J$ be two segments with disjoint interiors. 
Consider subgroups  $H,K\subset Homeo_+([0,1])$,   supported on $I$ and $J$, respectively.
Then the group generated by $H$ and $K$ is isomorphic to  $H\oplus K$.

\item Consider $f\in G$  and $I$ a fundamental domain of $f$. Let  $H\subset Homeo_+([0,1])$ be a 
subgroup of homeomorphsisms supported in $I$. 

Then the group generated by $H$ and $f$ is  $$<H,f>=\left(\bigoplus_\ZZ H\right)\rtimes \ZZ$$ 
where the factor $\ZZ$ is generated by $f$ and acts on  $\left(\bigoplus_\ZZ H\right)$ by conjugacy 
 as a shift of the $H$ factors.
\end{enumerate}

\end{lemm}
 
\section{Elementary groups} 
 
 In this section we  define rules for a family $\cS\subset Homeo_+(0,1])$ so 
 that the generated group is without crossing and  admits  $\cS$  as a topological basis. 
 
\subsection{Fundamental  systems and elementary groups}
 Recall that $f\in Homeo_+([0,1])$ is called simple if $[0,1]\setminus Fix (f)$
 has a unique connected component (whose closure  is the support $supp(f)$).  A simple homeomorphism
 $f$ is called \textbf{\emph{positive}} is $f(x)\geq x$ for all $x\in [0,1]$.
 
 Let $\cI nt([0,1])$ denote the set of segments of $[0,1]$
 \begin{defi} Consider a family $\cS=\{f,S_f,I_f\}\subset Homeo_+([0,1])\times \cI nt([0,1])\times \cI nt([0,1])$.  One says that $\cS$ 
 is a \textbf{\emph{fundamental system}}
 if: 
 \begin{itemize}
 \item for any $(f,S_f,I_f)\in \cS$, $f$ is a simple positive homeomorphism, $S_f=supp(f)$ and  $I_f\subset supp(f)$ 
 is a fundamental domain of $f$ 
 \item for any distinct $(f,,S_f,I_f)\neq (g,S_g, I_g)\in \cS$ one has 
 \begin{itemize}
  \item either $S_f$ and $S_g$ have disjoint interiors
  \item or $S_f\subset I_g$ or else $S_g\subset I_f$
 \end{itemize}
 \end{itemize}
 \end{defi}

 The aim of this section is: 
 
 \begin{prop}\label{p.elementaire} Let $\cS$ be a fundamental system.  
 Then, the group $G\subset Homeo_+([0,1])$ generated by
 the elements of $\cS$ is without crossing, and totally rational. 
 \end{prop}

 \begin{lemm}\label{l.elementairesuccessive} Let $\cS$ be a fundamental system and  $G\subset Homeo_+([0,1])$ be the group generated by
 $\cS$.  Assume that $\{a,b\}$ is a pair of successive fixed points of an element of $G$. 
 Then there are $f\in G$ and $(g,S_g,I_g)\in \cS$ so that $[a,b]=f(S_g)$.
 \end{lemm}
\begin{demo} 
Assumme it is not the case, and consider  a pair $\{a,b\}$ of successive fixed points which are not 
in the orbit of some $S_g$. Let $h=f_n^{\pm 1}\cdots f_1^{\pm 1}$, $f_i\in \cS$,  having $\{a,b\}$ as a pair of successive fixed
points. We chose $\{a,b\}$ and $h$ so that $n$ is minimal for these properties. 

Notice that $(a,b)$ is not disjoint from all the supports of the $S_{f_i}$, otherwize $h$ 
would be the identity
on $[a,b]$ contradicting the definition of $a,b$.

\begin{clai}
Consider $x\in(a,b)$ so that $h(x)\neq x$. The supports $S_{f_j}$ containing 
$x$ are totally ordered by the inclusion, by definition of a fundamental system. 
Let $i$ such that $S_{f_i}$ is the largest support $S_{f_j}, j=1\dots n$, containing $x$. 

Then $S_{f_i}$ contains $[a,b]$. 
\end{clai}
\begin{demo}For every $j$, one has $f_j(S_{f_i})=S_{f_i}$, because the support of $f_j$ is either disjoint from
$S_{f_i}$ or contained in it. Therefore the end points of $S_{f_i}$ are fixed points of all the $f_j$, hence of $h$.
The definition of successive fixed points of $h$ implies that $[a,b]\subset S_{f_i}$, concluding.
\end{demo}

One deduces

\begin{clai}
 $S_{f_j}\subset S_{f_i}$ for every $j$.
\end{clai}
\begin{demo}If $S_{f_j}$ is not contained in $S_{f_i}$, then $f_j$ is the identity map on $S_{f_i}$. 
Furthermore, $S_{f_i}$ is invariant under all the $f_k$, and contains $[a,b]$.  One deduces that $f_i$ is the 
identity map on the orbit of $[a,b]$ under the subgroup generated by the $f_k, k=1\dots n$.
Thus $\{a,b\}$ is still a pair of successive fixed points of the word obtained by deleting the 
letter $f_j$. This  contradicts the minimality of  $n$.  
\end{demo}

Now we splitt $\{1,...n\}=A\coprod B$ where $A$ are the indices $j$ so that $f_j=f_i$ and 
$B$ the other indices, 
so that, for $j\in B$ one has $$S_{f_j}\subset I_{f_i}.$$ 

Every element $f_j$, $j\in B$, acts as the identity on the orbit for $f_i$ of $\partial I_{f_i}$. 
One deduces 
\begin{clai}
Let $\alpha$ denote the sum  of the coefficient $\pm 1$ of the $f_{k}$, $k\in A$.  Then $\alpha=0$.
\end{clai}
\begin{demo}  Notice that , 
$h=f_i^\alpha$ on the orbit of $\partial I_{f_i}$.  If $\alpha\neq 0$, this implies that $h$ has no fixed point on $S_{f_i}$. 
Thus 
$\{a,b\}=S_{f_i}$ contradicting the definition of $\{a,b\}$. 
\end{demo}

Now, one can rewrite $h$ as a product of conjugated  of $f_j$, $j\in B$ by some power $f_i^{\beta_j}$. In other words there are 
$\beta_j\in \ZZ$, $j\in B$, so that 
$$h= \prod_{j\in B} f_i^{-\beta_j}f_jf_i^{\beta_j}.$$

Each of these conjugates $f_i^{\beta_j} f_j f_i^{-\beta_j}$ is supported in  $f_i^{\beta_j}(S_j)$ 
which is contained in the fundamental domain 
$f_i^{\beta_j}(I_{f_i})$. 

If all the $\beta_j$ are not equal, then the conjugates corresponding to different $\beta_j$ have disjoint 
interior of support, hence commutes. Furthermore, $\{a,b\}$ is contained in one of these intervals.  
One does not  change the pair of successive fixed point by deletting the terms corresponding 
to the other $\beta_j$. Once again, if the $\beta_j$ are not constant, one gets a smaller word, 
contradicting the minimality of $n$.

We can now assume that all the $\beta_j$ are equal to some $\beta$. So, $h$ is the conjugate by $f_i^{\beta}$ of the product of the $f_j, j\in B$.  So, $\{a,b\}$ is the image by $f_i^{\beta}$
of a pair $\{c,d\}$ of successive fixed points of the product of the $f_j, j\in B$. 
As the cardinal of $B$ is strictly smaller
than $n$, the minimality of $n$ implies that $\{c,d\}$ is the image by an element of $G$ of one of the 
$S_g$, $g\in \cS$.  One deduces the same property for $\{a,b\}=f_i^{\beta}(\{c,d\})$, 
getting a contraction with the definition of $\{a,b\}$.
 \end{demo}
 
 If the group  $G$ admits  a crossing,  this means by definition that there is a pair 
 of successive fixed points $\{a,b\}$ and an element 
 $g\in G$ with $g(\{a,b\})\cap (a,b)\neq \emptyset$. According to Lemma~\ref{l.elementairesuccessive}, 
 one can assume
 that $(a,b)$  is some $S_f$ with $f\in \cS$.  One concludes the proof ot Proposition~\ref{p.elementaire} 
 by showing:
 \begin{lemm}Let $\cS$ be a fundamental system and  $G\subset Homeo_+([0,1])$ be the group generated by
 the element of $\cS$. Given any $f\in \cS$ and any $g\in G$ either $g(S_f)$ and $S_f$ have disjoint 
 interior or $g(S_f)=S_f$.   
 \end{lemm}
 \begin{demo} Assume, arguing by absurd, that it is not ne case and consider 
 $g=g_n^{\alpha_n}\dots g_1^{\alpha_1}$, $g_i\in \cS$,   and $f\in\cS$ so that:
 \begin{itemize}
  \item $g_{i+1}\neq g_i$ for every $i\in\{1,\dots n-1 \}$, 
  \item $g(S_f)\neq S_f$ ,
  \item and $g(S_f)$ intersects the interior of $S_f$.
 \end{itemize}
 We chose $g,f$ so that $\sum_1^n |\alpha_i|$ is minimal for these properties.
 
If $S_{g_1}\subset S_f$ or if $S_{g_1}$ and $S_{f}$ have disjoint interiors, 
then $g_1(S_f)=S_f$.  In that case, one may delete $g_1$ contradicting the minimality of $n$. 

Thus, $S_f\subset I_{g_1}$, by definition of a fundamental system. 
So, $g_1^{\alpha_1}(S_f)\subset g_1^{\alpha_1}(I_{g_1})$ and its interior is disjoint from
$I_{g_1}$ but is contained in $S_{g_1}$. 

If $S_{g_1}\not\subset I_{g_2}$, as $g_2\neq g_1$, one gets that either $S_{g_2}$ is of 
interior disjoint from $S_{g_1}$, 
or $S_{g_2} \subset I_{g_1}$.  In both cases, $g_2$ is the identity map on $g^{\alpha_1}(g_f)$.  Thus,
 one may delete 
$g_2$, contradicting the minimality of the word. 

So $S_{g_1}\subset I_{g_2}$ and 
$g_2^{\alpha_2}g_1^{\alpha_1}(S_f)\subset g_2^{\alpha_2}(I_{g_2})\subset S_{g_2}$.  In particular, 
$g_2^{\alpha_2}g_1^{\alpha_1}(S_f)$ is disjoint from $S_f$. 

An easy induction proves that $S_{g_i}\subset I_{g_{i+1}}$ and 
$g_{i+1}^{\alpha_i+1}\cdots g_2^{\alpha_2}g_1^{\alpha_1}(S_f)$ is disjoint from $S_f$ for every $i$,
concluding.  
\end{demo}

\begin{defi} A group $G\subset Homeo_+([0,1])$ is called \textbf{\emph{an elementary group}} if it is generated by a
fundamental system.
\end{defi}

\begin{rema} If $H$ is topologically conjugate to an elementary group $G$, then $H$ is an elementary group. 
\end{rema}

\subsection{The topological dynamics and the topology of the fundamental systems}

 Next proposition explains that a elementary group, generated by a finite fundamental system,  is topologically determined 
 by the topological configuration 
 of the intervals $(S_f,I_f)$ of the funtamental systems.
 
 \begin{prop} Let $\cS$ and $\Si$ be two finite fundamental systems and let
 $G$ and $\Ga$ be the  groups generated by $\cS$ and $\Si$, respectively. 
 Assume that there is a 
 homeomorphism $h\colon [0,1]\to [0,1]$  and a bijection 
 $\varphi\colon \cS\to\Si$, $\varphi(f,S_f,I_f)=(\varphi(f), S_{\varphi(f)},I_{\varphi(f)})$ so that  
 for every $f$, 
 $$S_{\varphi(f)}=h(S_f)\quad\mbox{and}\quad I_{\varphi(f)}=h(I_f)$$

Then there is a homeomorphism $\tilde h\colon [0,1]\to [0,1]$ conjugating $G$ to $\Ga$: 
$$\Ga=\{\tilde h g \tilde h^{-1}, g\in G\}.$$
More precisely for every $(f,S_f,I_f)\in \cS$, $\tilde h f \tilde h^{-1}=\varphi(f)$. 
\end{prop}
\begin{demo} We argue by induction on the cardinal of $\cS$ and $\Si$.  If the cardinal is $1$, one just notes that
any to positive simple homeomorphism of $[0,1]$ are topologically conjugates. 

We assume now that the statement has be proved for any fundamental system of cardinal less or equal than $n$
and we consider  fundamental systems $\cS$ and $\Si$ of cardinal $n+1$. 

Let us write $\cS=\{(f,S_f,I_f)\}\cup \tilde \cS$ where  $S_f$ is 
a maximal interval in the (nested) family of supports, and 
$\Si=\{(\phi,S_{\phi}, I_{\phi})\}\cup \tilde \Si$ where $\phi=\varphi(f)$.

We consider $h_1$  so that $h_1 g h_1^{-1}= \varphi(g)$ for $g\in\tilde\cS$.  
Notice that $h_1(S_g)= h(S_g)=S_{\varphi(g)}$ for $g\in \tilde \cS$. 
 As a consequence, there is a homeomorphism $h_2$ which coincides
with $h_1$ on the union $S_{\tilde \cS}$ of the suports $S_g$, $g\in \tilde \cS$,   and with $h$ out of this union. 

Notice that $S_{\tilde \cS}\cap S_f\subset I_f$. Therefore $h_2(I_f)=h(I_f)=I_\phi$. 

Thus there is a unique homeomorphisms  $h_3\colon S_f\to S_\phi$ 
conjugating $f$ with $\phi$ and coinciding with $h_2$ on $I_f$.

The announced homeomorphism $\tilde h$ is the homeomorphism which coincides with $h_3$ on $S_f$, and with
with $h_2$ out of $S_f$. 
\end{demo}

 \subsection{Elementary groups of diffeomorphisms}
 If the homeomorphisms of a fundamental system are diffeomorphisms, the corresponding elementary group
 will be a group of diffeomorphisms.  
 
 \begin{prop} Any elementary group $G\subset Diff^1_+([0,1])$ is without hyperbolic fixed points. 
 \end{prop}
 
 \begin{demo} We consider a fundamental system $\cS=\{(f,S,f,I_f)\}$, with $f\in Diff^1_+([0,1])$.  Note that every diffeomorphism $f$ in $\cS$ 
 is, by definition of fundamental system, a simple
 diffeomorphism, and therefore has no hyperbolic fixed points. Let $G$ be the group generated by $\cS$

 Every $g\in G$ can be written as $g= f_n^{\alpha_n}\dots f_1^{\alpha_1}\in G$, $f_i\in \cS$, $f_{i+1}\neq f_i$,  $\alpha_i\in\ZZ$ (this
 presentation of $g$ may be not unique). 
 Arguing by absurd we assume that there is $g\in G$ having 
  hyperbolic fixed points.  We chose $g$ so that $\sum |\alpha_i|$ is the minimal number with this property. 
 
 An hyperbolic fixed point is isolated, so that it belongs to a pair of successive fixed points $\{a,b\}$ 
 (let assume it is $a$). 
 
 We consider $f_i$  so that the support $S_{f_i}$ is the largest of the supports $S_{f_j}$ containing $a,b$.  
 Then $S_{f_i}$ is invariant by all 
 the $f_j$.  If some $S_{f_j}$ is not contained in $S_{f_i}$,  one may delete the letter $f_j$ contradicting the 
minimality of $\sum |\alpha_i|$.
 
We  consider the sum  $\alpha$ of the $\alpha_j$ for which $f_j=f_i$ (recall that $S_{f_i}$ is the largest support). 
If $\alpha\neq 0$ then $g$ coincide with 
$f_i^\alpha$ on the orbit of $\partial(I_{f_i})$.  One  deduces that $\{a,b\}=\partial S_{f_i}$ and that the derivative of
$g$ at each end point is the same as the one of $f_i^{\alpha}$ which is $1$ (because $f_i$ is simple). 

So $\alpha=0$.  This allows us to rewrite $g$ as the product of conjugates of the $f_j\neq f_i$ by powers $f_i^{\beta_j}$.
These diffeomorphisms are supported on $f_i^{\beta_j}(I_{f_i})$ which have disjoint interiors. 
Exactly as in the proof 
of Lemma~\ref{l.elementairesuccessive}, one deduces that, if the $\beta_j$ are not all equal, then 
one may delete some
of the $f_j^{\alpha_j}$ contradicting the minimality of $\sum |\alpha_i|$.  

So the $\beta_j$ are all equal to some $\beta$ and one gets that $g$ is the conjugate by $f_i^{\beta}$ of 
the product of the $f_j^{\alpha_j}$, $f_j\neq f_i$. However, having a hyperbolic fixed point is 
invariant by conjugacy, so that one deduces that the product of the $f_j^{\alpha_j}$, $f_j\neq f_i$
has an hyperbolic fixed point.  This contradicts, once again, the minimality of $\sum |\alpha_i|$.
 \end{demo}

 One of the main result of this paper consits in proving that every elementary group of 
 diffeomorphisms of $[0,1]$ is 
 indeed $C^1$-close to identity.  This will be the aim of Section~\ref{s.C1}

\section{Free group}
The aim of this section is to prove Theorem~\ref{t.free} (assuming Theorem~\ref{t.jump}), that we recall below
\begin{theo}
 There is a subgroup $G\subset Diff^1([0,1])$, $C^1$-close to the identity, and 
  isomorphic to the free group $\FF^2$.
\end{theo}

Before that let us notice that the elementary fgroups do not contain any free group $\FF^2$.

\subsection{Finitely generated subgroups of elementary groups are solvable: proof of Proposition~ref{p.solvable2}}

In contrast we notice that elementary groups do not contain free groups:
\begin{prop}\label{p.solvable}
 Any elementary group $G$ generated by a finite fundamental system $\cS$ is solvable, with length bounded by the cardinal of $\cS$.
\end{prop}

Recall that a group is \textbf{\emph{solvable}} if the sequence $G_1=[G,G]$,\dots, $G_{n+1}=[G_n,G_n]$ statifies that there is 
$k$ so that $G_k=\{1\}$. The infimum of such $k$ is \textbf{\emph{the length $\ell(G)$}} of the solvable group $G$.

As a direct corollary one gets

\begin{coro}\label{c.solvable}
For any fundamental system $\cG$, the group $G$ generated by $\cG$ does not contain any non cyclic free group. 
\end{coro}
\begin{demo}[Proof of Corollary~\ref{c.solvable}] Assume there is a subgroup $<f,g>\subset G$ isomorphic to $F_2$. 
Notice that $f$ and $g$ are written as finite words in the 
generators in $\cG$ of $G$.  Thus the group $<f,g>$ is a subgroup of a 
group generated by a finite fundamental system.  This group is solvable, hence does not contain any subgroup isomorphic to $F_2$, 
contradicting the hypothesis.
\end{demo}

Notice that our argument above proved  Proposition~\ref{p.solvable2}.

We start the proof of Proposition~\ref{p.solvable} by showing:
\begin{lemm}\label{l.solvable} Fix and integer $k$.  Assume that $\{G_i\}_{i\in\NN}$ is a sequence of 
solvable groups of length  $\ell(G_i)\leq k$.  Then the abelian product 

$$G=\bigoplus_{i\in\NN} G_i$$
is solvable with length $\ell(G)\leq k$.
\end{lemm}
\begin{demo}Just notice that $[G,G]=\bigoplus_{i\in\NN}[G_i,G_i]$.
\end{demo}

\begin{demo}[Proof of Proposition~\ref{p.solvable}] We present a proof by induction on the cardinal of $\cS$.  
If this cardinal is $1$, the group
is a cyclic abelian group, hence is solvable of length $1$.  We assume now that Proposition~\ref{p.solvable} is proved for fundamental systems of 
cardinal less or equal to $n$. 
Let $\cS=\{(f_i, S_i,I_i), i\in\{1,\dots,n+1\}\}$ be a fundamental system.  Up to re-index the $f_i$, one may assume that 
$S_{n+1}$ is maximal for the inclusion among the $S_i$. 

Let $A\subset \{1,\dots, n\}$ denote the set of indices $i$ for which $S_i$ and $S_{n+1}$ have disjoint interiors, and 
$B=\{1,\dots, n\}\setminus A$ is the set of indices $j$ for which  $S_j$ is contained in $I_{n+1}$. 

First assume that $A$ is not empty.  Then $\cS_A=\{(f_i, S_i,I_i), i\in A\}$ and $\cS'=\{(f_j,S_j,I_j), j\in B\cup\{n+1\}\}$ are 
fundamental systems with cardinal $\leq n$.  Let $G_A$ and $G'$ the elementary groups generated by $\cS_A$ and $\cS'$, respectively. 
They are solvable groups of length bounded by $n$. Therefore, $G=G_A\oplus G'$ is solvable of length bounded by $n$, 
according to Lemma~\ref{l.solvable}, which concludes the proof in this case.

We assume now that $A$ is empty so that $B=\{1,\dots, n\}$.  Let denote $G_B$ the group generated by the fundamental system 
$\cS_B=\{(f_j, S_j,I_j), i\in B\}$.  One easily shows that $[G,G]$ is contained in $\bigoplus_{i\in \ZZ} G_{B,i}$ where 
$G_{B,i}=f_{n+1}^iG_Bf_{n+1}^{-1}$.  These groups are solvable of length $\ell(G_{B_i})\leq n$ by our induction hypothesis, so that 
$G$ is solvable of length $\ell(G)\leq n+1$ according to Lemma~\ref{l.solvable}.
 
\end{demo}

\subsection{Proof of Theorem~\ref{t.free}}
Let $\cA=\{a_i,i\in\NN\}$ be a countable set called alphabet. 
Let say that a word $\omega=\{\omega_i\}$ in  $n$ letters of the alphabet $\cA$ is 
\textbf{\emph{universal (among the  groups in $\cC^1_{id}$)}} if $\omega(f_1,\dots,f_n)=id$ for any 
$f_1,\dots, f_n\in Diff^1_+([0,1])$ for which the group $<f_1,\dots, f_n>$ is $C^1$-close to the identity. 

The length of such a word is $n$.  The word is reduced if $\omega_{i+1}\omega_i\neq 1$, and cyclically reduced 
if furthermore $\omega_n\omega_1\neq 1$. 

We will prove: 
\begin{prop}\label{p.free}
There is no  universal reduced non-trivial words. 
\end{prop}

Let us deduce the Theorem~\ref{t.free} from Proposition~\ref{p.free}. 

\begin{demo} Assume there is no universal reduced word.  In particular there is no universal word in 2 letters. 
Therefore, for any word $\omega$ in $2$ letters there is a pair $f_\omega, g_\omega$ so that the group 
$<f_\omega, g_\omega>$ is $C^1$-close to the identity and $\omega(f_\omega,g_\omega)\neq id$. 

One fix a sequence  $I_\omega\subset [0,1]$ of pairwize disjoint segments. For every $\omega$,  one choose
diffeomorphisms
$\tilde f_\omega, \tilde g_\omega$ supported on $I_\omega$, smoothly conjugated to $(f_\omega,g_\omega)$, and 
so that  $\tilde f_{\omega},\tilde g_{\omega}$ 
tend uniformly to the identity when the length of $\omega$ tends to $\infty$ (that is possible because  
$<f_\omega, g_\omega>$ is $C^1$-close to the identity).

One defines $f$ and $g$ as being $\tilde f_\omega$,$\tilde g_\omega$ on $I_\omega$ and the identity map out of 
the union of the $I_\omega$. One easily check that $f$ and $g$ are homeomorphisms. 
Now, $f$ and $g$ are diffeomorphisms because $\tilde f_{\omega},\tilde g_{\omega}$ 
tend uniformly to the identity. 

Finally the   group $<f,g>$ is isotopic to identity and 
 $\omega(f,g)\neq id$ for every reduced word $\omega$. 
\end{demo}

\subsection{No universal relation: proof of Proposition~\ref{p.free}}

Notice that, if $\omega$ is a universal word, then for each letter $a_i\in \cA$, the sum of the coeficients in 
$a_i$ of the $\omega_i$ vanishes.  Otherwize, $\omega(id,\dots,id,f,id, \dots,id)$
would be different from $id$, for $f\neq id$ at the $i^{th}$ position, contradicting the universality of 
the relation $\omega$.

Consider a letter appearing in $\omega$, say $\omega_1$. 
Then, one can write $\omega$ as a product of conjugates of the other letters by powers of the letter corresponding to  $\omega_1$.
 One replace each $\omega_1^{j}a_i\omega_1^{-j}$ by a new letter denoted $b_{i,j}$. 
 One gets a reduced words $\varphi(\omega)$ of 
 length  less or equal than $n-2$ (in the alphabet $\cB=\{b_{i,j}\}$). 
 
 Notice that a reduced word of length $\leq 2$ cannot be universal. 
 One concludes the proof of Proposition~\ref{p.free} and therefore of Theorem~\ref{t.free} by proving
 
 \begin{lemm}\label{l.free}
 If $\omega$ is universal then $\varphi(\omega)$ is universal. 
 \end{lemm}
 \begin{demo}Assume that $\varphi(\omega)$ is not universal.  Therefore, there is an interval $I$ and  
 $h_{i,j}\in Diff^1(I)$ so that $\varphi(\omega)(h_{i,j})\neq id$ 
 and the group generated by the $h_{i,j}$ is $C^1$-close to the identity. 
 
 One considers $f\in Diff^1_+([0,1])$, without hyperbolic fixed points, and so that $I$ is contained in the interior of a fundamental domain of $f$. 
 According to Theorem~\ref{t.jump}, the group generated by $f$ and the $h_{i,j}$ is $C^1$-close to the 
 identity. 
 
 One denotes by $g_i$ the diffeomorhisms which coincides with $f^{-j}h_{i,j}f^j$ on $f^{-j}(I)$, 
 for every $j$ for which $h_{i,j}$ is defined, and the identity out of the $f^{-j}(I)$. The group generated by the 
 $g_i$ is $C^1$-close to the identity. 
 
 Now $\omega(f,\{g_i\})$ is  a diffeomorphism which coincides with $\varphi(\omega)(h_{i,j})$ 
 in restriction to $I$, hence is not the identity map, contradicting the hypothesis on $\omega$.   
 \end{demo}

\section{Group extensions in the class $\cC^1_{id}$.}
The aim of this section is to prove  Theorem~\ref{t.jump}: given a group $G\subset Diff^1_+([0,1] $, $C^1$-close to the identity 
and supported in the interior of a fundamental domain of a diffeomorphism $f\in Diff^1_+([0,1])$, without hyperbolic 
fixed points, 
the group $<G,f>$ is $C^1$-close tot the identity.  We will prove her a slightly stronger version which will be used in the proof of 
Theorem~\ref{t.realisation}: 
\begin{theo}\label{t.jump2}
Let $f\in Diff^1_+([0,1])$ be a diffeomorphism without hyperbolic fixed points and $G\subset Diff^1_+([0,1])$ be a group supported on a fundamental domain 
$[x_0,f(x_0)]$.  Assume that there is a $C^1$-continuous path $h_t\in Diff^1_+([0,1])$, $t\in[0,1)$, so that 
\begin{itemize}
 \item for every $g\in G$ $$h_tgh_t^{-1}\overset{C^1}{\longrightarrow} Id,$$ and 
 \item  $h_t$ is supported on $[x_0,f(x_0)]$. 
\end{itemize}

Then the group $<f,G>$ generated by $f$ and $G$ is $C^1$-close to the identity.

\vskip 2mm

More precisely,
there is a $C^1$-continuous path $H_t\in Diff^1_+([0,1])$, $t\in[0,1)$, so that 
\begin{itemize}
 \item for every $g\in <f,G>$ $$H_tgH_t^{-1}\overset{C^1}{\longrightarrow} Id,$$ and 
 \item  $H_t$ is supported on the support of $f$ and $DH_t(0)=DH_t(1)=1$. 
\end{itemize}

\end{theo}
Notice that Theorem~\ref{t.jump} follows directly from Theorem~\ref{t.jump2}: if the support of $G$ is contained in the interior of the 
fundamental domain $(x_0,f(x_0))$ then given  $h^0_t\in Diff^1(supp(G)), t\in [0,1)$ realizing an isotopy by conjugacy of   $G$ to the identity, 
one easily build another isotopy by conjugacy $h_t\in Diff^1([0,1])$, supported on $[x_0,f(x_0]$.

Theorem~\ref{t.jump2} is the main technical result of this paper.  The proof is the aim of the whole section.  Let
us first present a sketch of proof. 
\subsection{Sketch of proof}

The proof  uses strongly arguments in \cite{Fa} which build explicitely conjugacies from 
a given diffeomorphisms to a neighborhood of the identity.  
Here we will need to come back to the contruction of \cite{Fa} for getting a simultanous conjugacy. 
For this reason it will be sometimes practical to use the following notation

\begin{nota}%\marginpar{Cette \\ notation\\ n'est \\ presque \\ pas utilise.\\L'utiliser \\ou l'effacer}
Given  $f,g\in Diff^1_+([0,1])$ and $(h_t)_{t\in[0,1)}$, $h_t\in Diff^1_+([0,1])$  a 
$C^1$-continuous path of diffeomorphisms, 
 the notation 
$$f\underset{h_t}{\rightsquigarrow}g$$ 
means that $(h_tfh_t^{-1})_{t\in[0,1)}$ is an \textbf{\emph{isotopy by 
conjugacy 
from $f$ to $g$}}, that is:  
$$h_tfh_t^{-1}\overset{C^1}{\underset{t\to 1}{\longrightarrow}} g. $$ 

One will write $f\rightsquigarrow g$ if there exists an isotopy by conjugacy from $f$ to $g$.
\end{nota}

\begin{demo}[Sketch of proof of Theorem~\ref{t.jump2}]
Let $G$ be a group $C^1$-close to the identity, supported in  a 
fundamental domain $[a,b]$ of 
$f\in Diff^1_+([0,1])$. 

Let $(h_t)_{t\in[0,1)}$ be a continuous path  of $C^1$-diffeomorphisms realizing
an isotopy by conjugacy from $G$ to $id$ and such that $h_t$ has derivative equal to $1$ at 
$a$ and $b$. We will  extend these diffeomorphisms $\{h_t\}_{t\in[0,1)}$
 to diffeomorphisms $\tilde h_t$  of $[0,1]$ in such a way that 
 $h_t f h_t^{-1}\overset{C^1}{\longrightarrow} f$ (Lemma \ref{l1.sansenlace}): in other word, for $t\to 1$, 
 the extensions $\tilde h_t$ almost commute with $f$. 
 
In this way, the impact 
induced on $f$ by the isotopy by conjugacy from $G$ to $id$ will be slight.

Let  $(\varphi_t)_{t\in[0,1)}$ be a continuous path  of $C^1$-diffeomorphisms for which   
 $\varphi_t f \varphi_t^{-1}\overset{C^1}{\longrightarrow} id$ (the existence of  $(\varphi_t)_{t\in[0,1)}$ is given \cite{Fa}); 
assume that   one can choose $\varphi_t$ so that, furthermore,  $\varphi_t$ is affine on $[a,b]$ 
or all $t\in[0,1)$.
Under this assumption, the conjugacy by $(\varphi_t)_{t\in[0,1)}$ 
will not affect  the $C^1$-distance of $h_tg h_t^{-1}$ to the identity for $g\in G$, 
and we can compose the two isotopies, 
by first conjugating by $\tilde h_t$ and then by $\varphi_t$: 
$$\forall g\in<G,f>, g\underset{\varphi_t\circ\tilde h_t}{\rightsquigarrow} id.$$ 

Indeed, we will ensure the existence of such $\varphi_t$ affine on the support of $G$ when the support of $G$ is 
contained in the interior of a fundamental domain of $f$.  When the support of $G$ is precisely one fundamental 
domain, we will weaken slightly this assumption, ensuring that the logarithm of the derivative of 
$\varphi_t$  is equicontinuous on $[a,b]$ (Lemma~\ref{l.step1induction}).  This will ensure that conjugating by $\varphi_t$ will have a bounded 
effect on the $C^1$-distance of $h_tg h_t^{-1}$ to the identity for $g\in G$. 

\end{demo}

As explained in this sketch of proof the two main tools for Theorem~\ref{t.jump} are Lemmas~\ref{l1.sansenlace} and 
\ref{l.step1induction} stated in next sections. 

\subsection{Background from \cite{Fa}}
We rewrite  \cite[Proposition 9]{Fa} 
in a form which will be more convenient here. 

If $h_0$ is a diffeomorphism supported in a fundamental domain of the diffeomorphism $f$, one gets a homeomorphism
$h$ commuting with $f$ by defining the restriction $h_n=h|_{I_n}$, $I_n=f^n(I_0)$ as  
$h_{n}=f_{n-1} h_{n-1} f^{-1}_{n-1}$ where $f_n$ is the restriction $f|_{I_n}$. The homeomorphism $h$ 
will be a diffeomorphism
if and only if $h_n$ tends to the identity map in the $C^1$ topology as $n\to\pm \infty$. That is not the case in general. 
However, if one only wants that $f$ and $h$ \emph{almost} commute, then one is allowed to modify slghtly the 
induction process, and one can do it guaranteeing this convergence to the identity.  
That was the aim of \cite{Fa}.  In Proposition~\ref{p.Fa1} and \ref{p.groupesanspointfixeenlace} below, 
we renormalized the intervals $I_n$ so that then $f_n$ appear as diffeomorphisms of $[0,1]$.

\begin{prop}\label{p.Fa1}
Let $(f_{n})_{(n)\in[0,1)\times\N}$ be a sequence of diffeomorphisms of
$[0,1]$ such that   $(f_{n})_{n\in\N}$ converges to the identity in the $C^1$-topology when $n$ tends to infinity 

Let  $(h_{0})$ be a 
diffeomorphism of $[0,1]$ such that, $Dh_{0}(0)=1=Dh_{0}(1)$.
Fix  $\varepsilon>0$ .

\vskip 2mm

 Then, there exists a sequence $(\psi_{n})_{(n)\in\N}$ of $C^1$-diffeomorphisms of $[0,1]$

such that : 
\begin{itemize}
\item $D\psi_{n}(0)=1=D\psi_{n}(1)$ for all $n\in\N$ ;
\item  $\Vert f_{n}\circ\psi_{n}-f_{n}\Vert_{1}<\varepsilon$ for all $n\in\N$ ;
\item the sequence $(h_{n})_{n\in\N}$ of $C^{1}$-diffeomorphisms of $[0,1]$ , defined by induction as 
$h_{0}$ and 
\\$h_{n}=f_{n-1}\psi_{n-1}h_{n-1}f_{n-1}^{-1}$ if $n\in\N^{*}$ satisfies: 
 $$\exists N>0, \forall n\geq N, h_n=id.$$
\end{itemize}
\end{prop}

For Theorem~\ref{t2.sansenlace} we will need the version with parameters, also due to \cite{Fa}, 
of this proposition.  Let us state it below:

\begin{prop}\label{p.groupesanspointfixeenlace}
Let $(f_{t,n})_{(t,n)\in[0,1)\times\N}$ be a collection of diffeomorphisms such that : 
\begin{itemize}
\item for all $n$, $(f_{t,n})_{t\in[0,1)}$ is a $C^{1}$-continuous path in $Diff^1_+([0,1])$ so that 
$Df_{t,n}(0)$ and $Df_{t,n}(1)$ do not depend on $t\in [0,1)$;
\item for all $t\in[0,1)$, $(f_{t,n})_{n\in\N}$ converges to the identity  in the $C^1$-topology, when $n$ tends to infinity.
 \end{itemize}

Let  $(h_{t,0})_{t\in[0,1)}$ be a $C^1$continuous path of 
diffeomorphisms of $[0,1]$ such that, for all $t\in[0,1)$, one has $Dh_{t,0}(0)=1=Dh_{t,0}(1)$.
Let $(\varepsilon_t)_{t\in[0,1)}$ be continuous path of strictly positive real numbers.

\vskip 2mm

 Then, there exists a collection $(\psi_{t,n})_{(t,n)\in[0,1)\times\N}$ of $C^1$-diffeomorphisms of $[0,1]$ %and a continuous path $(\tilde{\varepsilon}_t)_{t\in[0,1)}$
%of strictly positive real numbers 
such that : 
\begin{itemize}
\item $D\psi_{t,n}(0)=1=D\psi_{t,n}(1)$ for all $(t,n)\in[0,1)\times\N$ ;
\item %$\Vert \psi_{t,n}-id\Vert_{1}<\tilde{\varepsilon}_t$ for all $(t,n)\in[0,1)\times\N$, where $\tilde{\varepsilon}_t$ is such that, for all $n\in\N$, 
$\Vert f_{t,n}\circ\psi_{t,n}-f_{t,n}\Vert_{1}<\varepsilon_t$ for all $(t,n)\in[0,1)\times\N$ ;

\item Let $(h_{t,n})_{n\in\N,t\in [0,1)}$ be the collection of $C^{1}$-diffeomorphisms of $[0,1]$, defined by induction from  $h_{t,0}$ by
$$h_{t,n}=f_{t,n-1}\psi_{t,n-1}h_{t,n-1}f_{t,n-1}^{-1} \mbox{  if } n\in\N^{*}.$$ 

Then, for every $t\in[0,1)$, there is $N_t>0$, increasing with $t\in[0,1)$, so that 
$$h_{t,n}=f_{t,n-1}\psi_{t,n-1}h_{t,n-1}f_{t,n-1}^{-1}=id, \quad  \forall n\geq N_t.$$ 
\item $(\psi_{t,n})_{t\in[0,1)}$ is $C^1$-continuous for all $n\in\N$ ;
\item $(h_{t,n})_{t\in[0,1)}$ is $C^1$-continuous for all $n\in\N$.
\end{itemize}
\end{prop}

\subsection{Diffeomorphisms almost commuting with $f$ and prescribed in a fundamental domain}

\begin{lemm}\label{l1.sansenlace}
Let us consider $f\in\mc{D}iff_+^1([0,1])$, $I=[x,f(x)]$ a 
fundamental domain of $f$, $(a,b)$ the connected component of $[0,1]\setminus{Fix}(f)$ containing $x$, 
and $(h_t)_{t\in[0,1)}$ a $C^1$-continuous path of 
$C^1$-diffeomorphisms of $[0,1]$ supported on $I$.

Then, for all continuous path $(\varepsilon_t)_{t\in[0,1)}$ of strictly positive real numbers,
there exists a $C^1$-continuous path $(\tilde{h}_t)_{t\in[0,1)}$ 
of $C^1$-diffeomorphisms of $[0,1]$ such that, for all $t\in[0,1)$,:
\begin{itemize}
\item  $\tilde{h}_t$ coincides with the identity map on a neighbourhood of $a$ and $b$ ;
\item  the support of $\tilde h_t$ is contained  in the orbit for $f$ of the support of $h_t$:
$$\mathrm{Supp}(\tilde{h}_t)\subset \bigcup\limits_{n\in\Z}f^n(\mathrm{Supp}(h_t));$$ 
\item $\tilde h_t$ coincides with $h_t$ on the fundamental domain $I$: 
$$\restriction{\tilde{h}_t}{I}=h_t;$$ 
\item $\Vert\tilde{h}_tf\tilde{h}_t^{-1}-f\Vert_1<\varepsilon_t$.
\end{itemize}
\end{lemm}

The proof of Lemma~\ref{l1.sansenlace} consists in pushing  $h_t$ by $f$ in the iterates 
$f^n(I)$ of the fundamental 
domain. In each of these fundamental domains, one applies a small pertubation so that 
the diffeomorphism obtained in $f^n(I)$ becomes closer to $id$. 
%This enables to obtain a diffeomorphism which will quite be $C^1$ at $0$ and $1$. 
%For example, concerning the positive iterates, if $h_t=h_{t,n}$ has been defined on $I_n=f^n(I)$, then 
%$h_t$ will be defined on $I_{n+1}$ as $f\circ \psi_{t,n} h_{t,n} f^{-1}$ where 
%$f\circ \psi_{t,n}$ is a small perturbation of $f$.   

%\begin{prop}\label{p1.sansenlace}
%cas des suites croissantes
%\end{prop}
%\begin{prop}\label{p2.sansenlace}
%cas des suites décroissantes
%\end{prop}

\begin{demo}[Proof of Lemma \ref{l1.sansenlace}]

Let $(\varepsilon_t)_{t\in[0;1]}$ be a continuous path of strictly positive real numbers converging to $0$.
We  denote by $(f_n)_{n\in\N}$ the sequence of $C^1$-diffeomorphisms of $[0,1]$ defined by: for all $n\in\N$, $f_n$ is 
the normalization of the diffeomorphism $\restriction{f}{[f^{n}(x);f^{n+1}(x)]}$, 
that is : $f_n$ is obtained by conjugating 
$\restriction{f}{[f^{n}(x), f^{n+1}(x)]}$ by the affine maps from $[f^{n}(x),f^{n+1}(x)]$ and $[f^{n+1 }(x),f^{n+2}(x)]$to $[0,1]$. 
Notice that, as $f$ is $C^1$, the sequence 
$(f_n)_{n\in\N}$ converges to $id$ when $n$ tends to $\infty$, with respect to the $C^1$-topology.

One considers then $h_{t,0}$ as being the normalization of 
$\restriction{{h}_t}{[x,f(x)]}$ on the interval
$[0,1]$. In particular, the equality $Dh_{t,0}(0)=1=Dh_{t,0}(1)$ is satisfied. 

Proposition \ref{p.groupesanspointfixeenlace}, 
asserts that there exists a collection $(\psi_{t,n})_{t,n\in[0,1)\times\N}$ of diffeomorphisms of
$[0,1]$ %and a continuous path $(\tilde{\varepsilon}_t)_{t\in[0,1)}$ of strictly positive real numbers 
such that :
\begin{itemize}
\item $(\psi_{t,n})_{t\in[0,1)}$ is a $C^1$-continuous path for all $n\in\N$ ;
\item for all $(t,n)\in[0,1)\times\N$, $D\psi_{t,n}(0)=1=D\psi_{t,n}(1)$
\item %for all $(t,n)\in[0,1)\times\N$, one has :
%$\Vert\psi_{t,n}-id\Vert_{1}<\tilde{\varepsilon}_t$, where $\tilde{\varepsilon}_t$ is such that 
for all $(t,n)\in[0,1)\times\N$, one has : $\Vert f_n\circ\psi_{t,n}-f_n\Vert_{1}<\varepsilon_t$ ;
\item for all $t\in[0,1)$, the sequence of diffeomorphisms of $[0,1]$ defined 
\\by : $\left\lbrace\begin{array}{lll}
         h_{t,0} &  & \\
         h_{t,n} & = & f_{n-1}\psi_{t,n-1}h_{t,n-1}f_{n-1}^{-1} \mbox{for all } n\in\N^*
\end{array}\right.$ is stationnary, equal to $id$ for all $n\in\N$ great enough.
\item for all $n\in\N$, the path $(h_{t,n})_{t\in[0,1)}$ is a $C^1$-continuous path.
\end{itemize}
One gets a similar result and a similar $C^1$-continuous collection $(h_{t,n})_{(t,n)\in[0,1)\times(-\N)}$ by considering  
the negative iterates of $f$.

Consider now the $C^1$-continuous path $({h}_t)_{t\in[0,1)}$ of $C^1$-diffeomorphisms of $[0,1]$ defined by :
\begin{itemize}
\item $\restriction{{h}_t}{[f^n(x),f^{n+1}(x)]}$ is conjugated to $h_{t,n}$ by the affine map from $[f^n(x),f^{n+1}(x)]$ to $[0,1]$,
for  all $n\in\ZZ$ ;
%\item $\restriction{\tilde{h}_t}{[f^{-n-1}(x);f^{-n}(x)]}$ is the normalization of $h_{t,m_0-n}$ as a diffeomorphism of $[f^{-n-1}(x);f^{-n}(x)]$ for $n\geqslant m_0+1$ ;
\item ${h}_t=id$ on the complement of $\bigcup\limits_{n\in\Z}f^n([x,f(x)])$.
\end{itemize}

By  straightforward calculations using  $\Vert f_n\circ\psi_{t,n}-f_n\Vert_1<\varepsilon_t$ 
one gets $$\Vert D(h_{t,n+1}f_nh_{t,n}^{-1})-Df_n\Vert_1<\varepsilon_t,$$ from which follows 
$\Vert h_{t,n+1}f_nh_{t,n}^{-1}-f_n\Vert_1<\varepsilon_t$, thus 
$\Vert h_{t}fh_{t}^{-1}-f\Vert_1<\varepsilon_t$, concluding the proof.

\end{demo}

\subsection{Conjugacy  to the identity prescribed in a fundamental domain}\label{s.prescribeconjugacy}
The aim of this section is :

\begin{lemm}\label{l.step1induction} Let $f\in Diff^1_+([0,1])$ be a diffeomorphism without  
hyperbolic fixed point, and let $[a,b]$, $b=f(a)$, be a fundamental domain of $f$. 
There exists a $C^1$-continuous path $(\alpha_t)_{t\in[0,1)}$ of $C^1$-diffeomorphisms of $[0,1]$ so that: 
\begin{itemize}
 \item $D\alpha_t(0)=1=D\alpha_t(1)$ for all $t\in[0,1)$.
 \item $(\alpha_t)_{t\in[0,1)}$ has equicontinuous $\mathrm{Log}$-derivative on $[a,b]$: 
$$ \begin{array}{c}
 \forall \varepsilon>0, \exists \delta>0, \forall t\in[0,1), \forall x,y\in [a,b] \mbox{ so that }
 |x-y|\leq \delta
 \mbox{ one has: } \\ |\log D\alpha_t(x)-\log D\alpha_t(y)|<\varepsilon  \end{array}$$
 \item $f\underset{\alpha_t}{\rightsquigarrow}id$
\end{itemize}

\end{lemm}

The proof is a variation on the proof of the main result in \cite{Fa}:
\begin{theo}[\cite{Fa}]
Given $f\in Diff^1_+([0,1])$ without fixed points in $(0,1)$, given any continuous pathes $0<a_t<b_t<1$, $t\in[0,1)$,
 given any $C^1$-continuous path ${g_t}_{t\in[0,1)}$, where $g_t\in Diff^1_+([0,1])$ is a diffeomorphism without fixed point in 
 $(0,1)$ which coincides with $f$ on $[0,a_t]$ and on $[b_t,1]$,
there is a $C^1$-continuous path $h_t\in Diff^1_+([0,1])$, $t\in[0,1)$, 
 so that, for every $t$, $h_t$ coincides
 with the identity on a neighborhood of $0$ and of $1$, and the $C^1$ 
 distance $\|h_t fh_t^{-1}-g_t\|^1$ tends to $0$ as $t\to 1$.
\end{theo}

  Let us sketch the proof of \cite{Fa}, so that we will explain the modification we need here. 
  
\begin{demo}[Sketch of proof of \cite{Fa}]
 As $f$ and $g_t$ coincides on $[0,a_t]$, there is a unique diffeomorphism $\hat h_t$ of $[0,1)$
which is the identity map in a neighborhood of $0$, and conjugating the restriction $f|_{[0,1)}$ to $g_t|_{[0,1)}$.
\cite{Fa} chooses $h_t$ so that it coincides with $\hat h_t$ out of an arbitrarily small neighborhood of $1$.
The idea is that, in a neighborhood of $1$, $f$ anf $g_t$ coincide so that $\hat h_t$ commute with $f$. 
One concludes as in the proof of Lemma~\ref{l1.sansenlace}: by using 
Proposition~\ref{p.groupesanspointfixeenlace} one can modify $\hat h_t$ slowly in the successive fundamental domains
of $f$ in order to get a diffeomorphism $h_t$ coinciding with $\hat h_t$ out of a small neighborhood of $1$, 
with the identity map in a smaller neighborhood of $1$ and almost comuting with $f$ on $[b_t,1]$.
 
 \end{demo}
 
Let us now modify  slightly the proof of \cite{Fa}. 
Consider  point $x_t, y_t\in [a_t,b_t]$ varying continuously with $t\in[0,1)$.

Let  $\varphi_t\colon [x_t, f(x_t)]\to [y_t,g_t(y_t)]$, $t\in[0,1)$ be a $C^1$-continuous path of
 diffeomorphisms satisfying  $$D\varphi_t(f(x_t)Df(x_t)=Dg_t(y_t)D\varphi_t(x_t).$$
 Then there is a unique $C^1$-diffeomorphism $\tilde h_t\colon(0,1)\to (0,1)$ conjugating $f$ to $g_t$ and 
 coinciding with $\varphi_t$ on $[x_t,f(x_t)]$.  As before, as $f$ and $g_t$ coincide on $[0,a_t]$ and $[b_t,1]$,
 one gets that $\tilde h_t$ commutes with $f$ in a neighborhood of $0$ and of $1$. One concludes as before: 
 one can modify $\tilde h_t$ slowly in the successive fundamental domains
of $f$ and $f^{-1}$ in order to get a diffeomorphism $h_t$ coinciding with $\tilde h_t$ out of a 
small neighborhood of $0$ and of $1$, 
with the identity map in a smaller neighborhoods of $0$ and $1$ and almost commuting with $f$ on $[0,a_t]$ and 
on $[b_t,1]$.

 Summarizing, this proves: 

\begin{theo}\cite{Fa}\begin{itemize}\label{t.conjugaisonprescrite}
 \item given $f\in Diff^1_+([0,1])$ without fixed points in $(0,1)$ so that $f(x)-x>0$ on $(0,1)$
 \item given any continuous pathes $0<a_t<b_t<1$ $t\in[0,1)$, 
 \item given any continuous pathes $x_t,y_t\in[a_t,b_t]$, $t\in[0,1)$
 \item given any $C^1$-continuous path ${g_t}_{t\in[0,1)}$, where $g_t\in Diff^1_+([0,1])$ is a diffeomorphism without fixed point in 
 $(0,1)$   so that  that $g_t$ coincides with $f$ on $[0,a_t]$ and on $[b_t,1]$
 \item given any $C^1$-continuous path $\varphi_t\colon [x_t, f(x_t)]\to [y_t,g_t(y_t)]$  so that 
  $$D\varphi_t(f(x_t)Df(x_t)=Dg_t(y_t)D\varphi_t(x_t).$$
 
\end{itemize}
 Then there is a $C^1$-continuous path $h_t\in Diff^1_+([0,1])$, $t\in[0,1)$, 
 so that 
 \begin{itemize}
  \item for every $t$, the diffeomorphism $h_t$ coincides
 with the identity on a neighborhood of $0$ and of $1$
 \item $h_t$ coincides with $\varphi_t$ on $[x_t,f(x_t)]$
 \item the $C^1$  distance $\|h_t fh_t^{-1}-g_t\|_1$ tends to $0$ as $t\to 1$.
 \end{itemize}
\end{theo}

According to Theorem~\ref{t.conjugaisonprescrite}, Lemma~\ref{l.step1induction} is now a direct consequence of the following 
lemma:

\begin{lemm}\label{l.conjugaisonprescrite} Let $f\in Diff^1_+([0,1])$ without hyperbolic fixed points,  
without fixed points in $(0,1)$ 
so that $f(x)-x>0$ on $(0,1)$, and 
$[a,b]$, $b=f(a)$, be a fundamental domain of $f$.   Then:
\begin{itemize}
 \item there is a $C^1$-continuous path $g_t$, $t\in[0,1)$, $g_t\in Diff^1_+([0,1])$, so that:
\item $g_t$ is without fixed point in $(0,1)$
\item  there are continuous pathes $0<a_t< b_t<1$ so that $g_t$ coincides with $f$ on $[0,a_t]\cup[b_t,1]$
\item $g_t\underset{t\to 1}{\overset{C^1}{\longrightarrow}}id$
 
\item  there is a $C^1$-continuous path of diffeomorphisms 
$\varphi_t\colon [a,f(a)]\to [a,g_t(a)]$ so that 
 $$D\varphi_t(f(a)Df(a)=Dg_t(a)D\varphi_t(a), $$  
\item and $(\varphi_t)_{t\in[0,1)}$ has equicontinuous $\mathrm{Log}$-derivative on $[a,f(a)]$. 
\end{itemize}
\end{lemm}
\begin{demo}[Proof of Lemma~\ref{l.step1induction}] Lemma~\ref{l.conjugaisonprescrite} and Theorem~\ref{t.conjugaisonprescrite} 
 imply that there is $h_t$
so that $h_t f h_t^{-1}$ is $C^1$-asymptotic to the isotopy $g_t$ which tends to the identity.  Furthermore, $h_t$ 
coincides with $\varphi_t$ on $[a,f(a)]$, hence has equicontinuous $\mathrm{Log}$-derivative on the fundamental domain 
$[a,f(a)]$, ending the proof.
\end{demo}

 Lemma~\ref{l.conjugaisonprescrite} announces the existence of two objects: the path 
 $g_t\underset{t\to 1}{\overset{C^1}{\longrightarrow}}id$, and the path $\varphi_t$ with equicontinous $\mathrm{Log}$-derivative.
 This suggestes a natural splitting of the proof in two  easy observations.

 \begin{lemm} Consider $f\in Diff^1_+([0,1])$ without hyperbolic fixed points,
 without fixed points in $(0,1)$ so that $f(x)-x>0$ on $(0,1)$, and a fundamental domain $[c,d=f(c)]$  of $f$. 
 
 Let $c<d_t\leq d$ be 
a continuous path so that $d_t\to c$ as $t\to 1$ and $\alpha_t>0$ be a continuous path with 
$\alpha_t\to 1$ as $t\to 1$. 

Then there is a $C^1$ continuous path of diffeomorphisms $g_t\in Diff^1_+([0,1])$, there are continuous pathes $0<a_t<b_t<1$ so that
\begin{itemize}
 \item $g_t(x)>x$ for $x\in(0,1)$
 \item $g_t$ coincides with $f$ on $[0,a_t]\cup[b_t,1]$
 \item $g_t(c)=d_t$
 \item $Dg_t(c)=\alpha_t$
 \item  $g_t\underset{t\to 1}{\overset{C^1}{\longrightarrow}}id$
\end{itemize}
 \end{lemm}
 \begin{demo}[Hint for a proof:]
In other words, $g_t$ is an isotopy of $g_0$ to the identity map, without creating new fixed points, prescribing the image anf the derivative 
at a point $c$ and requiring that $g_t$ coincides with $f$ in a small neighborhood of $0$ and $1$.  That is possible because we require that the image 
$d_t=g_t(c)$ tends to $c$, that the derivative $\alpha_t=Dg_t(c)$ tends to $1$, and because 
$Df(0)=Df(1)=1$ so that $f$ is arbritrarily $C^1$-close to the identity map in sufficiently small neighborhoods of $0$ and $1$.
\end{demo}

 \begin{lemm} Consider $f\in Diff^1_+([0,1])$ without hyperbolic fixed points,
 without fixed points in $(0,1)$, so that $f(x)-x>0$ on $(0,1)$.  
Let $[c,d]$, $d=f(c)$ be a fundamental domain of $f$. Then, there is a $C^1$-continuous path of diffeomorphisms
$\varphi_t\colon [c,d]\to [c,\varphi_t(d)]$, $t\in[0,1)$,  so that 
\begin{itemize}
 \item $\varphi_t(d)\to c$ as $t\to 1$;
 \item Let us denote $\alpha_t=\frac{Df(c)\cdot D\varphi_t(d)}{D\varphi_t(c)}$. Then $\alpha_t\to 1$ as $t\to 1$;
 \item $\log(D\varphi_t)$, $t\in[0,1)$, is equicontinuous. 
\end{itemize}
 \end{lemm}
 \begin{demo}
Notice that adding a constant to a 
function does not change the equicontinuity properties. As a consequence, one can compose each 
$\varphi_t$ by some affine map without changing the equicontinuity of the family $\log(D\varphi_t)$; 
furthermore composing by an affine map does not change the ratio 
$\frac{Df(c)\cdot D\varphi_t(d)}{D\varphi_t(c)}$. In other words, the first item is for free. 
 
Now, we choose some $\varphi_{t_0}$ so that  
$\frac{Df(c)\cdot D\varphi_t(d)}{D\varphi_t(c)}=1$, and for $t>t_0$ one chooses $\varphi_t$ 
as beeing the composition of $\varphi_{t_0}$ by some affine map. 
 
 \end{demo}

\subsection{Conjugacy by a equicontinous Log-derivative map}

\begin{lemm}\label{l4.sansenlace}
Let $(\alpha_t)_{t\in[0,1)}$ be a $C^1$-continuous path of $C^1$-diffeomorphisms
of $[0,1]$ with equicontinuous  Log derivative: $\{\log D\alpha_t\}_{t\in[0,1)}$ is equicontinuous. 

Then for every $\eta>0$, there is $\varepsilon>0$ such that, for all $g\in\mc{D}iff^1_+([0,1])$ satisfying
$\Vert g-id\Vert_{1}<\varepsilon$, one has : 
$$\Vert\alpha_tg\alpha_t^{-1}-id\Vert_{1}<\eta \mbox{ for all }t\in[0,1).$$

In particular, if $(g_t)_{t\in[0,1)}$ is a path of diffeomorphisms converging to $id$ when $t$ tends to $1$, 
then $$\alpha_tg_t\alpha_t^{-1}\limite{t\to 1}{C^1}id.$$
\end{lemm}

\begin{demo} Consider $x\in[0,1]$ and $y=\alpha_t^{-1}(x)$.  Then 
$$D(\alpha_tg\alpha_t^{-1})= \frac{D\alpha_t(g(y)}{D\alpha_t(y)}\cdot Dg(y).$$
By assumption, $|Dg(y)-1|<\varepsilon$.
Therefore, it is enough to check that $$\log(\frac{D\alpha_t(g(y)}{D\alpha_t(y)})=\log D\alpha_t(g(y))-\log D\alpha_t(y)$$
is uniformly bounded in function of $\varepsilon$, and that this bound tends to $0$ as $\varepsilon\to 0$. 
Notice that $|g(y)-y|<\epsilon$. Thus, the equicontinuity of 
$\log D\alpha$ provides the uniform bound of $\log(\frac{D\alpha_t(g(y)}{D\alpha_t(y)})$ in function of $\varepsilon$.
\end{demo}

\subsection{ Isotopy by conjugacy to the identity and perturbations}

\begin{defi}\label{d1.sansenlace} Let $\varepsilon_t>0$  and $\eta_t>0$, $t\in[0,1)$ be  
continuous pathes with 
$\varepsilon_t\to 0$  and $\eta_t\to 0$ as $t\to 1$.

A $C^1$-continuous path $(\psi_t)_{t\in[0,1)}$, $\psi_t\in Diff^1_+([0,1])$ is
\textbf{\emph{an $(\varepsilon_t)_{t\in[0,1)}$-robust isotopy by conjugacy of 
speed $(\eta_t)_{t\in[0,1)}$ from $f$ to $id$}} if, for all continuous path 
$(g_t)_{t\in[0,1)}$ satisfying $\Vert g_t-f\Vert_{1}<\varepsilon_t$, one has : 
$$\Vert \psi_tg_t\psi_t^{-1}-id\Vert_1<\eta_t.$$
\end{defi}

\begin{lemm}\label{l3.sansenlace}
Let $f$, $(\varphi_t)_{t\in[0,1)}$ be $C^1$-diffeomorphisms of $\R$ 
such that $\Vert\varphi_tf\varphi_t^{-1}-id\Vert_1\limite{t\to 1}{}0$; $\varphi_0=id$. 
For all continuous path $(\varepsilon_t)_{t\in[0,1)}$ of strictly positive real 
numbers converging to $0$, there exist a continuous path $(\eta_t)_{t\in[0,1)}$ of strictly 
positive real numbers converging to $0$ and a continuous map $r : [0,1)\longrightarrow[0,1)$, 
satisfying $r(0)=0$, $r(t)\limite{t\to 1}{}1$ such that, $(\psi_t=\varphi_{r(t)})_{t\in[0,1)}$ is
an $(\varepsilon_t)_{t\in[0,1)}$-robust isotopy by conjugacy of 
speed $(\eta_t)_{t\in[0,1)}$ from $f$ to $id$.
\end{lemm}

%\begin{rema}
%If $r$ is a continuous map satisfying the conclusions above, and if $(\varepsilon_t)_{t\in[0,1)}$ and $\eta_t$ are 
% decreasing, then any continuous map
%$r'(t) : [0,1)\longrightarrow [0,1)$ such that $r'(t)\limite{t\to 1}{}1$ and satisfying
%$r'(t)\geqslant r(t)$ fulfils also the conclusions of Lemma \ref{l3.sansenlace}.
%\end{rema}

We split the proof in two  lemmas.  The first ones just states that any isotopy by conjugacy is $\varepsilon_t$ robust, if one chooses
$\varepsilon_t>0$ small enough. 

\begin{lemm}\label{l.robust1} Consider $f\in Diff^1_+([0,1])$ and  a $C^1$-continuous path $(\varphi_t)_{t\in[0,1)}$, $\varphi_0=id$,
so that $\Vert\varphi_tf\varphi_t^{-1}-id\Vert_1\limite{t\to 1}{}0$. 

Denote $\mu_t=2\cdot\Vert \varphi_tf\varphi_t^{-1}-id\Vert_1$.  Then there is a continuous path $\nu_t>0$ so that for  all continuous path 
$(g_t)_{t\in[0,1)}$ satisfying $\Vert g_t-f\Vert_{1}<\nu_t$, one has : 
$$\Vert \varphi_tg_t\varphi_t^{-1}-id\Vert_1<\mu_t.$$
 
 In other words, the isotopy by conjugacy  $\varphi_t$ is $\nu_t$-robust of speed $\mu_t$.
\end{lemm}
\begin{demo}[Sketch of proof]  For every $t\in[0,1)$, one  needs to bound $\frac{| D\varphi_t(g(x))-D\varphi_t(f(x)|}{D\varphi_t(x)}$, for 
$|g(x)-f(x)|<\nu_t$,  uniformly in $x\in[0,1]$, by a constant $\tilde \mu_t$ depending in a simple continuous way on $\mu_t$.
As $D\varphi_t$ is bounded on $[0,1]$, one essentially needs to bound  (uniformly in $x$) 
$D\varphi_t(g(x))-D\varphi_t(f(x))$, for 
$|g(x)-f(x)|<\nu_t$. In other words, $\nu_t$ depends strongly on the continuity modulus $\delta_t$ of $\varphi_t$ for the constant $\hat \mu_t$, where 
$\hat \mu_t=$. 
$$|x-y|<\delta_t\Rightarrow D\varphi_t(x)-D\varphi_t(y)<\tilde\mu_t\cdot  \max_{x\in[0,1]}|D\varphi_t(x)|$$

The unique difficulty is to choose $0<\nu_t<\delta_t$ depending continuously on $t\in[0,1)$.  This is possible because the modulus of continuity of
a continuous function (of a compact metric set) associated to a given constant, depends lower-semi-continuously on the function. 
One concludes by noticing that, given a strictly positive lower semi continuous map $\delta_t\colon [0,1)\to \RR$, there is a strictly positive function 
$0<\nu_t<\delta_t$.
 
\end{demo}

For getting Lemma~\ref{l3.sansenlace} from Lemma~\ref{l.robust1}, one just  needs to apply the following simple observation
\begin{lemm}\label{l.robust2}
 Let $\varepsilon_t>0$ and $\nu_t>0$, $t\in[0,1)$, be continuous paths  so that $\varepsilon_t\to 0$. Then, 
 there is a continuous map $r\colon [0,1)\to[0,1)$, 
 $r(0)=0$ and $r(t)\to 1$ as $t\to 1$, and there is $0\leq t_0<1$ so that, for every $t_0\leq t<1$,  one has:
 $$\nu_{r(t)}>\varepsilon_t.$$
\end{lemm}

\begin{demo}[Proof of Lemma~\ref{l3.sansenlace}] Choose $\mu_t$, $\nu_t>0$ given by Lemma~\ref{l.robust1}, 
and $r(t)$ and $t_0$ given by Lemma~\ref{l.robust2}.
Thus, for all continuous path 
$(g_t)_{t\in[0,1)}$ satisfying $\Vert g_t-f\Vert_{1}<\varepsilon_t<$, for every $t\geq t_0$ one has 
$$\Vert \varphi_{r(t)} g_t\varphi_{r(t)}^{-1}-id\Vert_1<\mu_{r(t)}.$$ 
Notice that $t\to \mu_{r(t)}$ is continuous and tends to $0$ as $t\to 1$.

In otherwords, the choice $\eta_t=\mu_{r(t)}$ is convenient for $t\geq t_0$. 
One extend such $\eta_t$ for $t\in[0,t_0]$ by a simple compactness argument.  More precisely:
one chooses $\eta_t$, $t\in[0,1)$ so that :
\begin{itemize}
 \item for every $t\in[0,1)$, $\eta_t\geq \mu_{r(t)}$
 \item $\eta_t=\mu_{r(t)}$ for every $t$ close enough to $1$. 
 \item for every $t\in[0,t_0)$, 
 $$\eta_t \geq \max_{x\in [0,1]} Df(x) + \max \varepsilon_t + \max_{t\in[0,t_0], x\in[0,1]} D\varphi_{r(t)}(x)
 +\max_{t\in[0,t_0], x\in[0,1]} D \varphi_{r(t)}^{-1}(x)$$
 \item $t\mapsto \eta_t$ is continuous.
\end{itemize}
For this choice of $\eta$, $\varphi_{r(t)}$ is $\varepsilon_t$-robust 
with speed $\eta_t$, concluding the proof.
 
\end{demo}

\subsection{Group extensions in the class $\cC^1_{id}$: proof of Theorem~\ref{t.jump}}

We are now ready to prove Theorem~\ref{t.jump}: 

\begin{demo}[Proof of Theorem~\ref{t.jump}]Let $f$ be a $C^1$-diffeomorphism of $[0,1]$, 
without hyperbolic fixed points  and $I=[x,f(x)]$ a fundamental domain of $f$.
Let $G\subset \mc{D}iff^1([0,1])$ be a group of diffeomorphisms whose supports are included in $I$.
Assume $G$ is  $C^1$-close to $id$;  more precisely, we assume that there is a $C^1$-continuous path 
of diffeomorphisms $h_t$, $t\in[0,1)$, supported on $I$, which realizes an isotopy by conjugacy of the elements of $G$ 
to the identity.  One will prove that the group $\langle G,f\rangle$, generated by $f$ and the elements of $G$,
is $C^1$-close to the identity and admits an isotopy by conjugacy to the identity $H_t$, $t\in[0,1)$, so that $D(H_t)(0)=D(H_t(1)=1$. 
Indeed, $H_t$ will coincide with the identity in small neighbohoods of $0$ and $1$.

One begins by extending the path $(h_t)_{t\in[0,1)}$ to $[0,1]$ by Lemma \ref{l1.sansenlace},
in such a way that $\Vert h_t f (h_t)^{-1}-f\Vert_1<\varepsilon_t$, 
where $(\varepsilon_t)_{t\in[0,1)}$ is some continuous path of strictly positive real numbers 
converging to $0$, and $h_t$ coincide with the identity in a neighborhood of $0$ and $1$.

As explained in Section \ref{s.prescribeconjugacy}, and from Lemma \ref{l3.sansenlace},
one can choose an $(\varepsilon_t)_{t\in[0,1)}$-robust isotopy $(\alpha_t)_{t\in[0,1)}$ from
$f$ to $id$ which has equicontinuous $\mathrm{Log}$-derivative, and so that $\alpha_t$ coincides 
with the identity map insmall neighborhoods of $0$ and $1$..

Then, by definition of an $(\varepsilon_t)_{t\in[0,1)}$-robust isotopy, one has :
\[ \Vert \alpha_th_tfh_t^{-1}\alpha_t^{-1}-id\Vert_1\limite{t\to 1}{}0, \]
and, from Lemma \ref{l4.sansenlace}, one has also :
\[ \forall g\in G, \Vert \alpha_th_tgh_t^{-1}\alpha_t^{-1}-id\Vert_{1}\limite{t\to 1}{}0. \]

Thus $H_t=\alpha_th_t$ is the announced isotopy by conjugacy of $<f,G>$ to the identity. 
\end{demo}

\section{Isotopy to the identity of groups generated by a fundamental system}\label{s.C1}

The aim of this section is to prove Theorem ~\ref{t.realisation}: any group $G$ of diffeomorphisms of $[0,1]$ generated by 
a fundamental system  is $C^1$-close to the identity. We will prove  a slightly stronger version: 
the diffeomorphisms $f_n$ are not assumed to be simple. 

\begin{theo}\label{t2.sansenlace}
Let $(f_n)_{n\in\N}$ be a collection of $C^1$-diffeomorphisms of $\R$ with compact support 
and without hyperbolic fixed point and, for each $n\in\N$, let $I_n$ be a given fundamental domain of 
$f_n$ such that :
\\
for all $i<n$, \hspace{.5cm}
\begin{minipage}[t]{6cm}
\begin{itemize}
\item either $\mathrm{Supp}(f_n)\subset I_i$ ;
\item or $\mathrm{Supp}(f_i)\subset I_n$ ;
\item or $\mathring{\mathrm{Supp}}(f_n)\cap\mathring{\mathrm{Supp}}(f_i)=\varnothing$  
\end{itemize}
\end{minipage}
\\
(where 
$\mathring{\mathrm{Supp}}(f)$ denotes the interior of the support of $f$).
\\
\medskip
Then the group $\langle f_n,n\in\N\rangle$ generated by $(f_n)_{n\in\N}$ is isotopic by 
conjugacy to the identity.
\end{theo}

Let $f_n$, $n\in \NN$, be a collection of diffeomorphisms satisfying the hypotheses of Theorem~\ref{t2.sansenlace} and let denote
by $G$ the group generated by the $f_n$.  Therefore, $G$ is the increasing union of the groups 
$G_n=<f_0,\dots,f_n>$. According to Theorem~\ref{t.union}, 
if all the $G_n$ are $C^1$-close to the identity, $G$ is $C^1$-close to the identity. 

Therefore, Theorem~\ref{t2.sansenlace} is a straightforward consequence of Theorem~\ref{t.union} with the 
following finite version of Theorem~\ref{t2.sansenlace}: 

\begin{theo}\label{t.finitecase}
Let $N>0$ be an integer and  $(f_n)_{n\in \{0,\dots, N\}}$ be a collection of 
$C^1$-diffeomorphisms of $[0,1]$  without hyperbolic fixed point and, 
for each $n\in\N$, let $I_n$ be a given fundamental domain of 
$f_n$ such that :
\\
for all $i<n$, \hspace{.5cm}
\begin{minipage}[t]{6cm}
\begin{itemize}
\item either $\mathrm{Supp}(f_n)\subset I_i$ ;
\item or $\mathrm{Supp}(f_i)\subset I_n$ ;
\item or $\mathring{\mathrm{Supp}}(f_n)\cap\mathring{\mathrm{Supp}}(f_i)=\varnothing$  
\end{itemize}
\end{minipage}
\\
(where 
$\mathring{\mathrm{Supp}}(f)$ denotes the interior of the support of $f$).
\\
\medskip
Then the group $\langle f_0,\dots, f_N\rangle$  is $C^1$-close to the identity. More precisely there is a
$C^1$-continuous family $\{h_t\}_{t\in[0,1)}$, $h_t\in Diff^1_+([0,1])$,   supported on  
$\bigcup_0^N \mathrm{Supp}(f_i)$, so that 
$Dh_t(0)=Dh_t(1)=1$ and 
$$\forall i\in\{0,\dots, N\}, \quad  f_i \underset{h_t}{\rightsquigarrow} id $$
\end{theo} 
\subsection{Proof of Theorem~\ref{t.finitecase}}

One proves Theorem~\ref{t.finitecase} by induction on $N$.  
For $N=0$, this is precisely the main result of \cite{Fa}. Assume now that Theorem~\ref{t.finitecase} 
is proved for $N\geq 0$; we will prove it for $N+1$. 

Let $f_0,\dots, f_{N+1}$ 
be diffeomorphisms satisfying the hypotheses of Theorem~\ref{t.finitecase}. 
The supports $S_i,S_j$ of  $f_i,f_j$, $i\neq j$ either have  disjoint interiors
or are included one in a fundamental domain of the 
other. 
Consider $\cI\subset \{0,\dots, N+1\}$ be the indices for which $S_i$  is maximal for the inclusion.

\subsubsection*{First assume that $\cI$ contains more than $1$ element.} Then, 
for every $i\in \cI$ the collection $\{f_j, S_j\subset S_i\}$ satisfies the hypotheses of
Theorem~\ref{t.finitecase}
and contains strictly less element than $N+1$.  Therefore, the induction hypothesis provides continuous pathes
$h^i_t, t\in[0,1)$, supported on $S_i$, realizing an isotopy 
of all the $f_j$ with $S_j\subset S_i$ to the identity, and so that the derivatives at 
$0$ and $1$ are equal to $1$.
One defines the announced family $h_t$  as coinciding with $h^i_t$ on $S_i$, $i\in \cI$.

\subsubsection*{Assume now that $\cI$ contains a unique element}
Up to change the indexation, one may assume that $\cI=\{N+1\}$.  Thus, the group $G_N=<f_0,\dots,f_N>$ is 
supported in the fundamental domain $I_{n+1}$ of $f_{N+1}$.  

By the induction hypothesis, there is a $C^1$-continuous path $h^N_t$, $t\in[0,1)$, supported on $I_{N+1}$ and 
realizing an isotopy by conjugacy of the elements of $G_N$ to the identity. 

Thus $G_N$ and $f_{N+1}$ satisfiy the hypotheses of Theorem~\ref{t.jump2}, which provides the announced path 
$h^{N+1}_t$, concluding the proof.

\noindent Christian Bonatti, \'Eglantine Farinelli

\noindent {\small Institut de Math\'ematiques de Bourgogne\\
UMR 5584 du CNRS}

\noindent {\small Universit\'e de Bourgogne, Dijon 21004, FRANCE}

\noindent {\footnotesize{E-mail : bonatti@u-bourgogne.fr}}

\noindent {\footnotesize{E-mail : trotti@live.fr}}

\end{document}